\documentclass[12pt]{amsart}

\usepackage[margin=1in]{geometry}
\usepackage[bookmarks, bookmarksdepth=2, colorlinks=true, linkcolor=blue, citecolor=blue, urlcolor=blue]{hyperref}
\usepackage{enumerate, amsmath, amsthm, amsfonts, amssymb, mathrsfs, graphicx, paralist, stmaryrd}
\usepackage[usenames, dvipsnames]{color}
\usepackage[all]{xy}
\usepackage{enumitem}
\usepackage{tikz-cd}
\usepackage{tikz}

\newcommand{\bV}{\mathbf{V}}
\newcommand{\bS}{\mathbf{S}}
\newcommand{\bp}{\mathbf{p}}
\newcommand{\bF}{\mathbf{F}}
\newcommand{\bN}{\mathbf{N}}
\newcommand{\bK}{\mathbf{K}}
\newcommand{\bL}{\mathbf{L}}
\newcommand{\bZ}{\mathbf{Z}}

\newcommand{\rR}{\mathrm{R}}
\newcommand{\rD}{\mathrm{D}}

\newcommand{\fS}{\mathfrak{S}}
\newcommand{\fm}{\mathfrak{m}}

\newcommand{\cA}{\mathcal{A}}
\newcommand{\cB}{\mathcal{B}}

\newcommand{\cF}{\mathcal{F}}
\newcommand{\cT}{\mathcal{T}}

\renewcommand{\phi}{\varphi}
\renewcommand{\emptyset}{\varnothing}

\newcommand{\lw}{{\textstyle \bigwedge}}

\makeatletter
\def\Ddots{\mathinner{\mkern1mu\raise\p@
\vbox{\kern7\p@\hbox{.}}\mkern2mu
\raise4\p@\hbox{.}\mkern2mu\raise7\p@\hbox{.}\mkern1mu}}
\makeatother

\DeclareMathOperator{\Div}{Div} 
\DeclareMathOperator{\coker}{coker}

\DeclareMathOperator{\cone}{Cone}

\DeclareMathOperator{\ext}{Ext}
\DeclareMathOperator{\reg}{reg}

\DeclareMathOperator{\Sym}{Sym}

\DeclareMathOperator{\Tor}{Tor}

\DeclareMathOperator{\len}{len}

\DeclareMathOperator{\Hom}{Hom}

\DeclareMathOperator{\Mod}{Mod}
\newcommand{\id}{\mathrm{id}}

\newcommand{\pol}{\mathrm{pol}}
\newcommand{\Pol}{\mathbf{Pol}}

\newcommand{\tors}{\mathrm{tors}}

\DeclareMathOperator{\FI}{FI}

\DeclareMathOperator{\VI}{VI}
\DeclareMathOperator{\Sh}{\bf{\Sigma}}
\DeclareMathOperator{\De}{\bf{\Delta}}
\DeclareMathOperator{\Rep}{Rep}

\DeclareMathOperator{\frob}{Frob}
\DeclareMathOperator{\Fec}{Vec}

\DeclareMathOperator{\Inc}{Inc}
\DeclareMathOperator{\init}{init}
\DeclareMathOperator{\colim}{colim}
\DeclareMathOperator{\fgen}{fg}
\DeclareMathOperator{\gen}{gen}
\DeclareMathOperator{\lf}{lf}
\DeclareMathOperator{\chark}{char}
\DeclareMathOperator{\kgdim}{KG.dim}

\newcommand{\GL}{\mathbf{GL}}

\makeatletter
\@addtoreset{equation}{section}
\makeatother

\numberwithin{equation}{section}
\newtheorem{theorem}[equation]{Theorem}

\newtheorem{proposition}[equation]{Proposition}
\newtheorem{lemma}[equation]{Lemma}
\newtheorem{corollary}[equation]{Corollary}

\theoremstyle{definition}
\newtheorem{rmk}[equation]{Remark}
\newenvironment{remark}[1][]{\begin{rmk}[#1] \pushQED{\qed}}{\popQED \end{rmk}}
\newtheorem{eg}[equation]{Example}

\newtheorem{defn}[equation]{Definition}

\makeatletter
\renewcommand{\thesubsection}{%
  \ifnum\c@subsection<1 \@arabic\c@section
  \else \thesection.\@arabic\c@subsection
  \fi
}
\makeatother

\title[$\GL$-algebras in positive characteristic I]{$\GL$-algebras in positive characteristic I: The  exterior algebra}

\author{Karthik Ganapathy}
\address{Department of Mathematics, University of Michigan, Ann Arbor, MI}
\email{\href{mailto:karthg@umich.edu}{karthg@umich.edu}}
\thanks{}
\urladdr{\url{http://www-personal.umich.edu/~karthg/}}
\usetikzlibrary{%
  matrix,%
  calc,%
  arrows%
}


\date{}

\begin{document}
\begin{abstract}
We study the category of $\GL$-equivariant modules over the infinite exterior algebra in positive characteristic. Our main structural result is a shift theorem {\`a} la Nagpal. Using this, we obtain a Church--Ellenberg type bound for the Castelnuovo--Mumford regularity. We also prove finiteness results for local cohomology. 
\end{abstract}
\maketitle
\section{Introduction}\label{s:intro}
 Over the last decade, multiple authors have proved interesting asymptotic and uniformity results in fields ranging from topology \cite{cef15fi,bg18orbitconf, mw19higher} and number theory \cite{ps17finite, mnp20sta} to algebraic statistics \cite{dra10fin, bd11egb2f, de15abtree, mn21equhilb} and commutative algebra \cite{sno13delta,lau18syz,ess19gen}.
 A unifying feature of these seemingly unrelated works is the prevalence of highly symmetric infinite-dimensional objects.
 The class of \textit{$\GL$-algebras}, i.e., rings which admit an (appropriate) action of the infinite general linear group $\GL$ plays a central role in this endeavour. 
 Snowden \cite{sno20spe, sno21stable}, Sam--Snowden \cite{ss16gl, ss19gl2}, and other authors \cite{nss16deg2,bdes21geo} have extensively studied $\GL$-algebras over a field of characteristic zero;
 their work has already found a diverse array of applications. 
 
 On the contrary, $\GL$-algebras in positive characteristic have hardly been studied despite being intimately connected to other ubiquitous algebraic structures like twisted commutative algebras. 
 We fill this gap in the literature and initiate the systematic study of $\GL$-algebras in positive characteristic. In this paper, we study the infinite exterior algebra in detail; in the sequel \cite{ganglp}, we investigate the infinite polynomial ring. In positive characteristic, there are many complications in the polynomial ring (like Frobenius-twisted ideals) that disappear in the exterior algebra, making the latter easier to analyze. 

Throughout the paper, we fix $k$ an algebraically closed field of characteristic $p > 0$. Let $\GL = \bigcup \GL_n(k)$ be the infinite general linear group, and let $\bV = \bigcup k^n$ be the defining representation of $\GL$. Let $R$ be the exterior algebra of $\bV$ regarded as an algebra object in $\Rep^{\pol}(\GL)$, and similarly, let $S$ be the symmetric algebra of $\bV$. The category of $R$-modules (resp.~$S$-modules) in $\Rep^{\pol}(\GL)$ will be denoted by $\Mod_R$ (resp.~$\Mod_S$). In this paper, we provide a detailed description of $\Mod_R$. Over a field of characteristic zero, Sam--Snowden \cite{ss16gl} have studied the algebraic and homological properties of $\Mod_R$ and $\Mod_S$. We make many of their results characteristic independent. For instance, we prove:

\begin{theorem}\label{thm:finreg}
The Castelnuovo--Mumford regularity of a finitely generated $R$-module is finite.
\end{theorem}
We emphasize that in characteristic zero, the analogue of Theorem~\ref{thm:finreg} is true for both $R$ and $S$. In characteristic $p$, this result fails for the polynomial ring $S$: the ideal generated by the $p$\textsuperscript{th} power of the variables is a $\GL$-stable ideal with infinite regularity. 

Theorem~\ref{thm:finreg} is almost a formal consequence of our main structural result (see Section~\ref{ss:schurder} for the definition of the Schur derivative):
\begin{theorem}[Shift Theorem]
Let $M$ be a finitely generated $R$-module and $\Sh$ be the Schur derivative. The $R$-module $\Sh^n(M)$ is flat for $n \gg 0$.
\end{theorem}

The shift theorem is known for $R$-modules in characteristic zero by work of Nagpal \cite{nag15fi}. 

\begin{remark}
Nagpal's version of the shift theorem is for $\FI$-modules in all characteristics. An $\FI$-module is a functor from $\FI$, the category of finite sets with injections, to the category of $k$-vector spaces. In characteristic zero, the category of $\FI$-modules is equivalent to the category $\Mod_S$ (by Schur--Weyl duality) which is equivalent to the category $\Mod_R$ (using the ``transpose functor" \cite[Section~7]{ss12tca}). All these equivalences fail in positive characteristic.
\end{remark}

We now explain the key technical novelty of our paper. Our proof of the shift theorem is motivated by Li--Yu's simplification \cite{ly17fi} of Nagpal's work. Crucial to Li--Yu's proof is the fact that for an $\FI$-module $M$, the generation degree of $\De(M) = \coker(M \to \Sh(M))$ is exactly $t_0(M) - 1$. This property fails for $R$-modules: in Corollary~\ref{cor:dfrobzero}, we show that for a Frobenius twisted $\GL$-representation $W$, we have $\De(R \otimes W)= 0$. 
Our main idea is to classify the $R$-modules for which $\De$ vanishes. 
We achieve this in Section~\ref{ss:deltavanishing} by controlling the $\GL$-representations that occur in such modules using the Steinberg tensor product theorem. In fact, we show that if $\De$ vanishes for a torsion-free $R$-module, then the module is (up to extensions) of the form $R \otimes W$ with $W$ Frobenius twisted. We also use these ideas in our forthcoming paper on $\Mod_S$ \cite{ganglp}.

The theory of $\GL$-algebras is more difficult in positive characteristic, but not merely due to the lack of semisimplicity. Instead, as the previous paragraph shows, the presence of Frobenius twisted representations is the primary obstacle to extending many results from characteristic zero to positive characteristic. Furthermore, recent work \cite{bdes21geo, sno21stable} about $\GL$-algebras in characteristic zero heavily exploit the presence of weight vectors of weight $(1^n)$ in an arbitrary $\GL$-representation; all the weights occurring in a Frobenius twisted representation are divisible by $p$ so we cannot use these methods. The Frobenius powers of an ideal also cause significant pathology. We \cite{ganglp} show that in the infinite polynomial ring $S$, the Frobenius powers of the maximal ideal behave like prime ideals in a suitable sense. In characteristic zero, the appropriate spectrum of $S$ only has two points (the zero ideal and the maximal ideal); in positive characteristic, it has infinitely many points. We hope to address these issues more systematically in the future.

\begin{remark}\label{rmk:histshift}
A $\VI$-module is a functor from $\VI$, the category of finite dimensional vector spaces over $\bF_q$ with injections, to the category of vector spaces over a field $L$. Nagpal \cite{nag19vi} also proved a shift theorem for $\VI$-modules in non-describing characteristic; i.e., when $\chark(\bF_q) \ne \chark(L)$. The aforementioned property about the shift functor also fails for $\VI$-modules. As a workaround, Nagpal proved that in non-describing characteristic, a $\VI$-module is flat, in a suitable sense, if and only if its local cohomology modules vanish. We prove a similar equivalence for $R$-modules (Proposition~\ref{prop:semiindiffdersat}), but obtain it as a corollary of the shift theorem. It would be interesting to see if the strategy we employ here can be used to reprove Nagpal's theorem.
\end{remark}

\subsection{Additional results on \texorpdfstring{$R$}{R}-modules}\label{ss:addn}
Apart from the shift theorem, we extend a number of other results about the exterior algebra to all characteristics.
\begin{itemize}
    \item We prove $\Mod_R$ is locally noetherian using Gr{\"o}bner theory (Theorem~\ref{thm:noetherianity}). 
    \item We show that the Krull--Gabriel dimension of $\Mod_R$ is one (Proposition~\ref{prop:kgdimone}).
    \item We obtain finiteness results for the local cohomology functors $\rR\Gamma$; i.e., prove that for a finitely generated $R$-module $M$, the $R$-module $\rR^i\Gamma(M)$ is finite length for all $i$, and vanishes for sufficiently large $i$ (Theorem~\ref{thm:finitelc}).
    \item We show that the derived category $\rD^b_{\fgen}(\Mod_R)$ is generated by the simple $R$-modules $L_{\lambda}$ and the flat $R$-modules $R \otimes L_{\lambda}$ for arbitrary $\lambda$ (Theorem~\ref{thm:gendbmodr}).
    \item We explain how to make Theorem~\ref{thm:finreg} effective using work of Gan--Li \cite{gl20vi} (Remark~\ref{rmk:cebounds}). This mirrors a theorem of Church--Ellenberg \cite{ce17regularity} on $\FI$-modules. 
\end{itemize}

Not everything generalizes, however. In characteristic zero, Sam--Snowden \cite{ss16gl} proved that $\Mod_R^{\gen}$ and $\Mod_R^{\tors}$ are equivalent, and showed that every finitely generated $R$-module has finite injective dimension. Both these properties fail in our setting (Proposition~\ref{prop:gentorinequiv}), though remarkably the functor $\rR\Hom(M, -)$ still preserves the derived category $\rD^b_{\fgen}(\Mod_R)$ (Theorem~\ref{thm:finiteext}).

There are many more results known in characteristic zero; we are unaware if they continue to hold in positive characteristic. For example, Nagpal--Sam--Snowden \cite{nss18reg} showed that the regularity of an $R$-module can also be computed using the local cohomology functors. It would be interesting if one can prove this in positive characteristic as well (the analogous fact for $\bZ$-graded Gorenstein rings is classical). 

\subsection{Relation to other work}\label{ss:future}
This work fits in with the broader goal of equivariant commutative algebra wherein one studies $G$-equivariant modules over a $k$-algebra $A$.
We collect some of the relevant examples from the literature in Table~\ref{tab:otherwork}. Our paper suggests that the exterior algebra is a much more tractable object in positive characteristic compared to the polynomial ring. 

\subsubsection{$\GL$-algebras} 
Fix a positive integer $r$, and let $A = \lw(k^r \otimes \bV)$ and $B = \Sym(k^r \otimes\bV)$. In characteristic zero, it is easy to see that the two categories $\Mod_A$ and $\Mod_B$ are equivalent; Sam--Snowden \cite{ss19gl2} have studied the structural properties of $\Mod_B$. 
The results (and techniques) of our paper is strong evidence that in positive characteristic, at least $\Mod_A$ should parallel the characteristic zero story.

\begin{table}[b]\label{tab:otherwork}
    \begin{center}
    \begin{tabular}{|c|c|c|l|}
    \hline
    Group & Ring & $\chark(k)$ & Results  \\
    \hline
    $\GL$ & $\Sym(\bV)$ \& $\lw(\bV)$& $\chark 0$ & \cite{cef15fi,ss16gl,nag15fi}\\
    $\GL$ & $\Sym(\bV^{\oplus n})$ \& $\lw(\bV^{\oplus n})$& $\chark 0$ & \cite{ss19gl2, ram17fid} \\
    $\fS_\infty$ & $\Sym(\bV)$ & any $\chark$ & \cite{ns20symsub,ns21symide}    \\
    $\bS\bp$ & $\Sym(\bV)$ & $\chark 0$ & \cite{ss22sp}\\
    $\GL$ & $\Sym(\bV)$ & $\chark p > 0$ &  \cite{ganglp} \\
    $\GL$ & $\lw(\lw^2(\bV))$ \& $\lw(\Sym^2(\bV))$ & $\chark 0$  & \cite{nss19skew} \\
    $\GL$ & $\Sym(\Sym^2(\bV)$ \& $\Sym(\lw^2(\bV))$ & $\chark 0$ & \cite{nss16deg2} \\
    $\GL$ & $\Sym(\bV \oplus \lw^2(\bV))$ & $\chark 0$ & \cite{ss22sp} \\
    \hline
    \end{tabular}
    \end{center}
    \caption{Other examples from equivariant commutative algebra}
\end{table}
\subsubsection{Invariant theory} Gandini \cite{gan18res, gan21deg} has studied equivariant resolutions over exterior algebras in characteristic zero by transferring well-established results about the polynomial ring to the exterior algebra using the transpose functor (see \cite[Section~7.4]{ss12tca}). In positive characteristic, the transpose functor does not exist, and moreover these results usually fail for the polynomial ring. It would be interesting to see whether one can directly prove analogues of Gandini's results in positive characteristic.

\subsubsection{The infinite symmetric group}
Nagpal--Snowden \cite{ns20symsub,ns21symide} have classified the $\fS_\infty$-stable ideals in the infinite polynomial ring. There are many remarkable results and open problems \cite{lnnr20pdim, lnnr21reg, murai20betti} about these ideals and their resolutions. The exterior algebra should be a good testing ground for these problems, being somewhat less complicated structurally. 
\subsubsection{Twisted commutative algebras}
A twisted commutative algebra (tca) is a commutative algebra object in $\Rep(\fS_{\star})$, the category of sequences of representations of $\fS_n$. We have a symmetric monoidal functor $\cF \colon \Rep^\pol(\GL) \to \Rep(\fS_\star)$ often referred to as the Schur functor. Given an object $W$ in $\Rep^\pol(\GL)$, the functor $\cF$ induces a functor $\cF_W \colon \Mod_{\Sym(W)} \to \Mod_{\Sym(\cF(W))}$. In characteristic zero, the functor $\cF$ is an equivalence, so the theory of tca's and $\GL$-algebras in characteristic zero agree. In positive characteristic, the Schur functor is not an equivalence. However, we are hopeful that the theory of $\GL$-algebras in positive characteristic will enable us to prove new results about tca's in positive characteristic.

\subsubsection{Categories with shift-like functors} 
Gan--Li \cite{gl19shift} have described an inductive machinery to prove results (like Noetherianity and finite regularity) about a module category equipped with a shift-like functor (see Subsection~1.9 of loc.~cit.). In their work, they assume $t_0(\De(M)) = t_0(M) - 1$ so it does not apply in our case. Nevertheless, developing an inductive machinery without this assumption on $\De(M)$ deserves serious consideration as more and more shift functors do not satisfy this condition (see also Remark~\ref{rmk:histshift}). 

\subsubsection{The infinite polynomial ring} 
As we remarked earlier, the infinite polynomial ring $S$ is more complicated in positive characteristic. In a forthcoming paper \cite{ganglp}, we provide a fairly comprehensive picture of $\Mod_S$. The structural results of this paper will be used to prove finiteness properties for resolutions of $S$-modules.

\subsection{Outline}
In Section~\ref{s:bg}, we recall the necessary background. In particular, we recall the notion of a polynomial representation in Section~\ref{ss:polrep}, and define $\GL$-algebras and related concepts in Section~\ref{ss:GLalg}. The Schur derivative and some of its properties are discussed in Section~\ref{ss:schurder}. We prove the noetherianity result (Theorem~\ref{thm:noetherianity}) and facts about torsion $R$-modules in Section~\ref{s:rmod}. Section~\ref{s:shift} contains the proof of the shift theorem with the key technical results appearing in Section~\ref{ss:deltavanishing}. In Section~\ref{s:fgrmod}, we use the shift theorem to obtain structural results about the derived category $\rD^b_{\fgen}(\Mod_R)$.

For the expert, we suggest first reading Section~\ref{ss:schurder}, Proposition~\ref{prop:flatequalssemi}, and then directly jumping to Section~\ref{s:shift}. Most of the other results in Section~\ref{s:bg} and Section~\ref{s:rmod} are as one would expect.
\subsection{Notation:}
\begin{description}[align=right,labelwidth=2.5cm,leftmargin=!]
\item[$k$] the base field, always algebraically closed of characteristic $p > 0$
\item[$\bV$] a fixed infinite dimensional $k$-vector space with basis $\{e_i\}_{i \ge 1}$
\item[$\GL$] the group of automorphisms of $\bV$ fixing all but finitely many of the basis vectors $e_i$
\item[$\Fec$] the category of $k$-vector spaces
\item[$\Rep^\pol(\GL)$] the category of polynomial representations of $\GL$
\item[$\Pol$] the category of strict polynomial functors $\Fec \to \Fec$
\item[$R$] the exterior algebra of $\bV$
\item[$L_{\lambda}$] the irreducible polynomial representation of $\GL$ with highest weight $\lambda$
\item[$\Sh$] the Schur derivative functor, which we also call the shift functor
\item[$\De$] the difference functor
\item[$-^{<n}$] the submodule generated by all elements of degree less than $n$
\item[$-^{(r)}$] the $r$-th Frobenius twist of a $\GL$-representation
\item[$t_i(-)$] the generation degree of $\Tor_i^R(-, k)$
\end{description}

\subsection*{Acknowledgements} I am grateful to Andrew Snowden for being generous with his ideas and providing numerous comments on a draft which greatly improved the exposition. Thanks are also due to Rohit Nagpal for pointing out helpful references, and Nate Harman for helpful discussions. The work was supported in part by NSF grants DMS \#1453893 and DMS \#1840234.

\section{Background}\label{s:bg}
\subsection{Polynomial representations}\label{ss:polrep}
We let $\GL = \bigcup_{n\geq 1} \GL_n(k)$, and $\bV = \bigcup_{n \geq 1} k^n$ be the standard representation of $\GL$. A representation of $\GL$ is \textit{polynomial} if it appears as a subquotient of a (possibly infinite) direct sum of tensor powers of $\bV$. We denote the category of all polynomial representations by $\Rep^\pol(\GL)$. It is a Grothendieck abelian category closed under tensor products whose objects are locally finite length. For a quick review of polynomial representations, we refer the interested reader to \cite[Section~1]{tot97gl}. 
\subsubsection*{Examples}The irreducible polynomial representations of $\GL$ are indexed by partitions of arbitrary size. For a partition $\lambda$, we let $L_\lambda$ be the irreducible representation with highest weight vector of weight $\lambda$, and we let $\bS_\lambda(\bV)$ be the Schur module corresponding to $\lambda$. The representation $L_\lambda$ is the socle of $\bS_\lambda(\bV)$. We note that $\bS_{(d)}(\bV) = \Sym^d(\bV)$, and $\bS_{(1^d)}(\bV) = \lw^d(\bV)$. Furthermore, we let $\Div^d(\bV)$ be the $d$-th divided power of $\bV$. For a partition $\lambda = (\lambda_1, \lambda_2, \ldots, \lambda_n)$, we let $\Div^{\lambda}(\bV) = \Div^{\lambda_1}(\bV) \otimes \Div^{\lambda_2}(\bV) \otimes \ldots \Div^{\lambda_n}(\bV)$, and define $\Sym^{\lambda}(\bV)$ similarly. The modules $\Div^{\lambda}(\bV)$ are projective objects in $\Rep^{\pol}(\GL)$, and any polynomial representation is a quotient of direct sums of $\Div^{\lambda}(\bV)$ (as $\lambda$ varies). Similarly, the representations $\Sym^{\lambda}(\bV)$ are injective objects in $\Rep^{\pol}(\GL)$.
\subsubsection*{Polynomial functors}Let $\Pol$ denote the category of strict polynomial endofunctors of $\Fec$ in the sense of Friedlander--Suslin \cite{fs97coh}. The category $\Rep^\pol(\GL)$ is equivalent to $\Pol$ \cite[Lemma~3.4]{fs97coh}. The equivalence is given by evaluating a polynomial functor $F$ on $\bV$. We let $\Phi \colon \Rep^\pol(\GL) \to \Pol$ denote the inverse. Under this equivalence, the Schur functor $\bS_\lambda$ in $\Pol$ is mapped to the polynomial representation $\bS_\lambda(\bV)$, and if $W$ is a polynomial representation of $\GL$, then $W \cong \Phi(W)(\bV)$. Every concept we define for polynomial representations can be transferred to $\Pol$ using this equivalence.

\subsubsection*{Frobenius twist} Let $\frob \colon \GL \to \GL$ denote the Frobenius map, which raises each entry of a matrix to its $p$-th power. The Frobenius map is a group homomorphism. Given a representation $W$ of $\GL$, the \textit{Frobenius twist} of $W$, denoted $W^{(1)}$, is the $\GL$-representation obtained by pulling $W$ back along this map. If $W$ is a polynomial representation, then $W^{(1)}$ is also a polynomial representation. For $r > 1$, we recursively define $W^{(r)} = (W^{(r-1)})^{(1)}$.

\subsubsection*{Grading}
A \textit{homogeneous} polynomial representation of \textit{degree $d$} is a polynomial representation of $\GL$ which has a (possibly infinite) filtration where the successive quotients are the simple objects $L_{\lambda}$ with $|\lambda| = d$. 
The degree is also compatible with the tensor product and Frobenius twist: if $V$ is homogeneous of degree $d$ and $W$ is homogeneous of degree $e$, then $V \otimes W$ is homogeneous of degree $d + e$, and $V^{(1)}$ is homogeneous of degree $pd$. 
Given an arbitrary polynomial representation $W$ of $\GL$, we can decompose it as $\bigoplus_{i \in\bN} W_i$ where $W_d$ is homogeneous of degree $d$ \cite[Lemma~2.5]{fs97coh}. Therefore, every polynomial representation has a canonical $\bN$-grading. We let $\Rep^\pol(\GL)_d$ be the full subcategory of $\Rep^\pol(\GL)$ on the homogeneous polynomial representation of degree $d$.

\subsubsection*{Steinberg tensor product theorem}A partition $\lambda = (\lambda_1, \lambda_2, \ldots, \lambda_n)$ is \textit{$p$-restricted} if $\lambda_i - \lambda_{i+1} < p$ for all $1 \leq i \leq n$ (here $\lambda_{n+1} = 0$). We state an important result that relates $p$-restricted partitions, Frobenius twists, tensor products, and irreducible representations, see e.g., \cite[Theorem~1.1]{kuj06stpt}.
\begin{theorem}\label{thm:steinberg}
Let $\lambda$ be a partition and write $\lambda = \sum_{i=0}^n p^i \lambda^i$ for $p$-restricted partitions $\lambda^0, \lambda^1, \ldots, \lambda^n$. Then, we have an isomorphism $L_\lambda \cong L_{\lambda^0} \otimes L_{\lambda^1}^{(1)} \otimes L_{\lambda^2}^{(2)} \ldots \otimes L_{\lambda^n}^{(n)}$.
\end{theorem}

\subsection{\texorpdfstring{$\GL$}{GL}-algebras}\label{ss:GLalg}
A \textit{$\GL$-algebra} is a commutative, associative, unital $k$-algebra equipped with an action of $\GL$ via algebra homomorphisms, under which it forms a polynomial representation. A \textit{skew $\GL$-algebra} is a skew-commutative associative, unital $\bZ/2$-graded $k$-algebra equipped with an action of $\GL$ via algebra homomorphisms under which it forms a polynomial representation.
We write $|A|$, when we want to forget the $\GL$-action on $A$ and think of it as an ordinary $k$-algebra. The $\GL$-action endows $|A|$ with a canonical $\bN$-grading. In a skew $\GL$-algebra, this canonical $\bN$-grading is unrelated to the $\bZ/2$-grading. Given a polynomial representation $W$, the algebra $\Sym(W)$ is a $\GL$-algebra and $\lw(W)$ is canonically a skew $\GL$-algebra.

\subsubsection*{\texorpdfstring{$A$}{A}-modules}Let $A$ be a (skew) $\GL$-algebra. We let $\Mod_A$ denote the category of $A$-modules in $\Rep^\pol(\GL)$. Concretely, an \textit{$A$-module} $M$ is an $|A|$-module, equipped with an action of $\GL$ that is compatible with the action on $A$ (i.e., $g(am) = g(a)g(m)$ for all $a \in A, m \in M$ and $g \in \GL$) such that $M$ is a polynomial representation. The action of $\GL$ endows $M$ with an $\bN$-grading, making $M$ an $\bN$-graded $|A|$-module. A \textit{$\GL$-ideal} of $A$ is a $\GL$-stable ideal of $|A|$. A \textit{submodule} of an $A$-module is a $\GL$-stable $|A|$-submodule. We denote by $A_{+}$ the ideal of positive degree elements of $A$. 

\subsubsection*{Finite generation}A $\GL$-algebra $A$ is \textit{finitely generated} if there exists a finite length subrepresentation $V \subset A$ such that the canonical map $\Sym(V) \to A$ is surjective. 
A skew $\GL$-algebra $A$ is \textit{finitely generated} if there exists finite length subrepresentations $V \subset A_{\overline{0}}$ and $W \subset A_{\overline{1}}$ such that the canonical map $\Sym(V) \otimes \lw(W) \to A$ is surjective.
Equivalently, the $k$-algebra $|A|$ is generated by the $\GL$ orbits of finitely many elements. An $A$-module $M$ is \textit{finitely generated} if there exists a finite length subrepresentation $V \subset M$ such that the canonical map $A \otimes V \to M$ is surjective. Equivalently, we can find finitely many elements in $M$ such that their $\GL$-orbits generate $|M|$ as an $|A|$-module.

Let $M$ be an $A$-module. The \textit{generation degree} of $M$ is the smallest integer $n \geq -1$ such that $M$ is generated by elements of degree $\leq n$ (or $\infty$ if no such integer exists). We say $M$ is \textit{generated in degree $n$} if $M$ is generated by its degree $n$ elements (in particular, we have $M_i = 0$ for $i < n$). 

\subsubsection*{Noetherianity}A (skew) $\GL$-algebra is \textit{noetherian} if every submodule of a finitely generated $A$-module is also finitely generated. It is \textit{weakly noetherian} if every $\GL$-ideal of $A$ is finitely generated.

\subsubsection*{Nakayama's Lemma} We have the following version of Nakayama's lemma. The usual proof of the (graded version of the) lemma applies.
\begin{lemma}\label{lem:nakayama}
    Let $A$ be a $\GL$-algebra and $M$ be an $A$-module. Let $V \to M$ be a map of $\GL$-representations such that the composition $V \to M \to M/A_{+}M$ is surjective. The map $A \otimes V \to M$, obtained by adjunction, is also surjective.
\end{lemma}

\subsubsection*{Projective and injective modules}For a projective polynomial representation $V$, the $A$-module $A \otimes V$ is projective. Since there are enough (finite length) projectives in $\Rep^\pol(\GL)$, the category $\Mod_A$ has enough (finitely generated) projectives. In fact, every $A$-module is a quotient of direct sums of the projective modules $\{A \otimes \Div^{\lambda}(\bV)\}_{\lambda}$. The category $\Mod_A$ has enough injectives as well since it is a Grothendieck abelian category (though the injectives may not be finitely generated).

\subsection{The Schur derivative}\label{ss:schurder} 
In this subsection, we define an endofunctor of $\Rep^\pol(\GL)$ called the \textit{Schur derivative}. We first define it for $\Pol$ and then define it for $\Rep^\pol(\GL)$ using the equivalence $\Phi$. 
Some properties of the Schur derivative are outlined in \cite[Section~6.4]{ss12tca} (under char $0$ assumptions). The Schur derivative and related notions also appear in \cite[Section~3]{bdde21univ} under the name ``linear approximation".

\subsubsection*{Schur derivative of polynomial functors}Let $F \colon \Fec \to \Fec$ be a polynomial functor. We define the Schur derivative of $F$, denoted $\Sh(F)$, to be the polynomial functor which on objects is defined by
\begin{displaymath}
\Sh(F)(V) = F(k \oplus V)^{[1]},
\end{displaymath}
where the superscript here denotes the subspace of $F(k \oplus V)$ on which $k^\star$ acts with weight~$1$. For a linear map $f \colon V \to W$, the induced map $\Sh(F)(f) \colon \Sh(F)(V) \to \Sh(F)(W)$ is given by restricting the map $F(\id_k \oplus f) \colon F(k \oplus V) \to F(k \oplus W)$ to the $1$ weight-space of the $k^\star$ action. It is easy to see that $\Sh(F)$ is a polynomial functor, and that the assignment $F \to \Sh(F)$ is functorial. 

\subsubsection*{Basic properties of the Schur derivative}The Schur derivative is an exact functor (because $k^\star$ is semisimple). It preserves finite length objects of $\Pol$. Furthermore, the Schur derivative is a categorification of the ordinary derivative (of polynomials) as it satisfies:
\begin{itemize}
    \item Additivity: $\Sh(F \oplus G) = \Sh(F) \oplus \Sh(G)$
    \item Leibniz rule: $\Sh(F\otimes G) = (\Sh(F) \otimes G) \oplus (F \otimes \Sh(G))$
    \item Chain rule: $\Sh(F\circ G) = (\Sh(F)\circ G) \otimes \Sh(G)$ 
    \item Kills ``$p$-th powers": $\Sh(F^{(1)}) = 0$
\end{itemize}
For the second property, we use the fact that for $k^\star$ representations $V$ and $W$, we have $(V \otimes W)^{[1]} = (V^{[1]} \otimes W) \oplus (V \otimes W^{[1]})$, and for the last property, we use the fact that all the weights occurring in a Frobenius twisted representation are divisible by $p$.

\subsubsection*{Schur derivative of a polynomial representation}Since $\Pol$ and $\Rep^\pol(\GL)$ are equivalent, we obtain an endofunctor of $\Rep^\pol(\GL)$, which we again denote by $\Sh$ and call the Schur derivative. For a polynomial representation $W$, we identify $\Sh(W)$ as a subspace of $\Phi(W)(k \oplus \bV)$, where $k^\star$ acts with weight $1$. We also identify $W$ is as a subspace of $\Phi(W)(k \oplus \bV)$ (where $k^\star$ acts trivially). The newly introduced basis vector of $k \oplus \bV$ will usually be denoted by $f_1$, or $y_1$ depending on the situation. The proof of the next lemma is easy, but demonstrates how we will work with the Schur derivative in the rest of the paper.
\begin{lemma}\label{lem:shiftofsym}
    For all $n \geq 1$, we have isomorphisms $\Sh(\Sym^n(\bV)) \cong \Sym^{n-1}(\bV)$, and $\Sh(\lw^n(\bV)) \cong \lw^{n-1}(\bV)$.
\end{lemma}
\begin{proof}
The vector space $\Sym^n(\bV)$ has basis indexed by degree $n$ monomials in $x_1, x_2, \ldots$. We shall identify $\Sh(\Sym^n(\bV))$ as a subspace of $\Sym^n(k \oplus \bV)$, where the newly introduced basis vector will be $y_1$.

We define a map $\Sym^{n-1}(\bV) \to \Sh(\Sym^n)(\bV) = \Sym^n(k \oplus \bV)^{[1]}$,
where
\begin{displaymath}
x_1^{i_1}x_2^{i_2}\ldots x_m^{i_m} \mapsto y_1 x_1^{i_1}x_2^{i_2}\ldots x_m^{i_m}.
\end{displaymath}
This is $\GL(\bV)$-equivariant ($\GL(\bV)$ acts trivially on $y_1$), and easily checked to be both injective and surjective. The proof for $\lw^n(\bV)$ is similar.
\end{proof}

\subsubsection*{Schur derivative of an $A$-module} Let $A$ be a $\GL$-algebra, and $M$ be an $A$-module. The $\GL$-representation $\Sh(M)$ is canonically an $A$-module, as we now explain. Firstly, we have a canonical $\GL$-equivariant map $A \to \Phi(A)(k \oplus \bV)$ induced by the canonical inclusion $\bV \to k \oplus \bV$. Now, the $\GL$-representation $\Phi(M)(k \oplus \bV)$ is a $\GL(k \oplus \bV)$-equivariant module over $\Phi(A)(k \oplus \bV)$. By restriction of scalars, it is a $\GL$-equivariant module over $A$. The subspace $\Sh(M)$ of $\Phi(M)(k \oplus V)$ is stable under the action of $\GL$ as well as $|A|$, so $\Sh(M)$ is an $A$-module. To summarize, the Schur derivative induces a functor $\Sh \colon \Mod_A \to \Mod_A$. For example, given a polynomial representation $V$, we have $\Sh(A \otimes V) = (A \otimes \Sh(V)) \oplus (\Sh(A) \otimes V)$.

\subsection{Semi-induced modules}\label{ss:semiinduced}
Throughout this subsection, we fix a (skew) $\GL$-algebra $A$ such that $A_0 = k$. 

An $A$-module is \textit{induced} if it is isomorphic to $A \otimes W$ for some polynomial representation $W$. An $A$-module is \textit{semi-induced} if it has a finite filtration where the successive quotients are induced modules. The terminology is motivated by similarly defined objects for $\FI$-modules (and related algebraic structures) \cite{nag15fi}.

Given an $A$-module $M$, we let $\Tor_i^A(M, -)$ denote the left derived functors of the right exact functor $M \otimes_A -$. An $A$-module $M$ is \textit{flat} if the functor $M \otimes_A -$ is exact. An $A$-module $M$ is flat if and only if the functors $\Tor_i^A(M, N) = 0$ for all $i > 0$ and all $A$-modules $N$. Finally, it is easy to see that the $|A|$-modules $|\Tor_i^A(M, N)|$ and $\Tor_i^{|A|}(|M|, |N|)$ are isomorphic.

The main result of this subsection is a characterization of semi-induced modules using $\Tor$.
\begin{proposition}\label{prop:flatequalssemi}
    Let $M$ be a finitely generated $A$-module. The following are equivalent:
    \begin{enumerate}[label=(\alph*)]
        \item $M$ is semi-induced,
        \item $M$ is flat,
        \item $\Tor_i^A(M, k) = 0$ for all $i > 0$, and
        \item $\Tor_1^A(M, k) = 0$.
    \end{enumerate}
\end{proposition}

We let $t_i(M) = \max\deg \Tor_i^A(M, k)$ as a polynomial representation with the convention that $\max\deg(0) = -1$. By Nakayama's lemma (Lemma~\ref{lem:nakayama}), the module $M$ is generated in degrees $\le n$ if and only if $t_0(M) \le n$. We denote by $M^{< n}$ the $A$-submodule of $M$ generated by all elements of degree $< n$. If $t_0(M) = n$, the quotient $M/M^{<n}$ will be generated in degree $n$; we implicitly use this throughout the paper.

\begin{lemma}\label{lem:semiequalsflat}
   Let $M$ be a semi-induced $A$-module and $N$ be any $A$-module. Then $\Tor^A_i(M, N) = 0$ for $i > 0$. 
\end{lemma}
\begin{proof}
The induced module $A \otimes V$ is flat since the functor $ (A \otimes V) \otimes_A - $ is isomorphic to $V \otimes_k -$ which is exact. By d\'evissage, we see that semi-induced modules are also flat.
\end{proof}

\begin{lemma}\label{lem:relationsinlowdeg}
    Let $M$ be an $A$-module generated in degree $n$ such that $t_1(M) \leq n$. The natural map $A \otimes M_n \to M$ is an isomorphism.
\end{lemma}
\begin{proof}
The map $A \otimes M_n \to M$ is surjective since $M$ is generated in degree $n$. The kernel $K$ of this map must be supported in degrees $> n$. By Lemma~\ref{lem:semiequalsflat}, the module $\Tor_1^A(A\otimes M_n, k) = 0$. Therefore, the long exact sequence of $\Tor_{\bullet}^A(-, k)$ gives us the exact sequence
\begin{displaymath}
0 \to \Tor_1^A(M, k) \to \Tor_0^A(K, k) \to \Tor_0^A(A \otimes M_n, k) \to \Tor_0^A(M, k).
\end{displaymath}
The rightmost map is an isomorphism, so the first map is also an isomorphism. Now, since $t_1(M) \leq n$, we have $t_0(K) \le n$ which implies $K$ is zero since $K_i = 0$ for $i \le n$. Therefore the map $A \otimes M_n \to M$ is an isomorphism, as required.
\end{proof}

\begin{lemma}\label{lem:subofsemi}
     Let $M$ be a finitely generated $A$-module such that $\Tor_1^A(M,k) = 0$ and $t_0(M) = n$. The module $M/M^{<n}$ is induced, and $\Tor_1^A(M^{< n},k )= 0$. 
\end{lemma}

\begin{proof}
    The module $M/M^{<n}$ is generated in degree $n$. We have the long exact sequence
    \begin{displaymath} 
    \Tor_2^A(M/M^{<n}, k) \to \Tor_1^A(M^{<n},k) \to \Tor_1^A(M, k) \to \Tor_1^A(M/M^{<n},k) \to \Tor_0^A(M^{<n},k),
    \end{displaymath}
    from which we see that $t_1(M/M^{<n}) < n$. So by Lemma~\ref{lem:relationsinlowdeg}, the module $M/M^{<n}$ is an induced module. In particular, the module $\Tor_2^A(M/M^{<n}, k) = 0$ by Lemma~\ref{lem:semiequalsflat}, which in turn implies that $\Tor_1^A(M^{<n}, k)$ vanishes (from the long exact sequence above), as required.
\end{proof}

\begin{proof}[Proof of Proposition~\ref{prop:flatequalssemi}]
   (a) implies (b) is Lemma~\ref{lem:semiequalsflat}. (b) implies (c), and (c) implies (d) is trivial. We now prove (d) implies (a). Using lemma~\ref{lem:subofsemi}, we see that the module $M/M^{<t_0(M)}$ is induced, and $\Tor_1^A(M^{<t_0(M)}, k) = 0$. So by induction on generation degree, we have that $M^{<t_0(M)}$ is semi-induced, and therefore $M$ is also semi-induced, as required.
\end{proof}

\begin{corollary}\label{cor:semises}
Let $0 \to M_1 \to M_2 \to M_3 \to 0$ be a short exact sequence of finitely generated $A$-modules.
\begin{enumerate}[label=(\alph*)]
    \item If $M_1$ and $M_3$ are semi-induced, then so is $M_2$.
    \item If $M_2$ and $M_3$ are semi-induced, then so is $M_1$.
\end{enumerate}
\end{corollary}
\begin{proof}
The analogous results for flat modules can be obtained using the long exact sequence of $\Tor$. We obtain the corollary for semi-induced modules using Proposition~\ref{prop:flatequalssemi}.
\end{proof}

\subsection{Serre subcategories and Serre quotients}\label{ss:serre}
We now review Serre quotients; see \cite{gab62serre} or \cite[Chapter~15]{fa12cat} for general background. 

A \textit{Serre subcategory} $\cB$ of an abelian category $\cA$ is a full abelian subcategory of $\cA$ that is closed under extensions, and taking subquotients. Given a Serre subcategory $\cB \subset \cA$, we let $T \colon \cA \to \cA/\cB$ denote the exact functor to the Serre quotient $\cA/\cB$ and $S \colon \cA/\cB \to \cA$ denote the right adjoint to $T$ (the section functor), if it exists. 

If $\cA$ and $\cB$ are Grothendieck abelian categories, then the section functor exists (see e.g., \cite[Theorem~15.11]{fa12cat}). Furthermore, under these assumptions, we also have a functor $\Gamma \colon \cA \to \cB$ which takes an object $M$ to the maximal subobject $\Gamma(M)$ of $M$ contained in $\cB$. The functor $\Gamma$ is right adjoint to the inclusion functor $i\colon \cB \to \cA$, so left exact; its right derived functors are the local cohomology functors.

All the results we use about local cohomology will be from \cite[Section~4]{ss19gl2}. Many of the results from loc.~cit.~are stated under the assumption that the $\cB$ satisfies Property (Inj), i.e., injective objects in $\cB$ remain injective in $\cA$. We will verify this in cases of interest.

\section{Basic results on \texorpdfstring{$R$}{R}-modules}\label{s:rmod}
For the rest of the paper, we let $R = \lw(\bV)$ be the exterior algebra. This is a skew $\GL$-algebra. We let $\fm$ denote the ideal of all positive degree elements of $R$. In degree $i$, the $\GL$-representation $R_i$ is the $i$-th exterior power $\lw^i(\bV)$ which is irreducible of highest weight $(1^i)$. The irreducibility of $R_i$ implies that the only nonzero proper $\GL$-ideals of $R$ are the ideals $\fm^r$ for $r \ge 1$. Clearly, the ascending chain condition holds for $\GL$-ideals of $R$, i.e., the skew $\GL$-algebra $R$ is weakly noetherian. In Section~\ref{ss:noetherianity}, we show that $R$ is noetherian. In Section~\ref{ss:torsion}, we define the notion of a torsion $R$-module (which makes sense only in the infinite variable setting).

\subsection{Noetherianity}\label{ss:noetherianity} 
In this subsection, we show $\Mod_R$ is locally noetherian. We will prove a stronger statement that given a finitely generated $\Inc(\bN)$-equivariant $R$-module $M$, every $\Inc(\bN)$-stable $R$-submodule of $M$ is also finitely generated. Our proof uses the following theorem of Cohen \cite{co67laws} (see also \cite[Corollary~6.16]{nr19fioi}):
\begin{theorem}\label{thm:cohen}
Let $M$ be a finitely generated $\Inc(\bN)$-equivariant module over the infinite polynomial ring $S = k[x_1, x_2, \ldots, ]$. Every $\Inc(\bN)$-stable $S$-submodule of $M$ is finitely generated. 
\end{theorem}
For the rest of the subsection (except the proof of Theorem~\ref{thm:noetherianity}), we forget the action of $\GL$ on $R$ and $S$.

\subsubsection*{Notation} Let $\Inc(\bN)$ denote the \textit{increasing monoid}, i.e., the monoid of increasing injective functions from $\bN$ to itself. Let $B$ denote the $k$-algebra $S/(x_1^2, x_2^2, \ldots, x_n^2, \ldots)$. We let $P_n$ denote the $R$-module $R \otimes \bV^{\otimes n}$, and let $Q_n$ denote the $B$-module $B \otimes \bV^{\otimes n}$ for all $n \geq 0$.

\subsubsection*{Monomial submodules}An element of $P_n$ is \textit{monomial} if it is of the form 
\begin{displaymath}
m = (x_{i_1}\wedge x_{i_2} \ldots \wedge x_{i_r}) \otimes (e_{j_1} \otimes e_{j_2} \ldots \otimes e_{j_n})
\end{displaymath} with $i_1 > i_2 > \ldots > i_r$.
We also have a similar notion of monomials for the $B$-module $Q_n$.

A submodule $M$ of the $R$-module $P_n$ (resp.~the $B$-module $Q_n$) is \textit{monomial} if it is generated by all the monomials it contains.

\subsubsection*{Bijection between monomial submodules of $P_n$ and $Q_n$} 
We define a map $\Psi$ from the set of monomials of $P_n$ and the set of monomials of $Q_n$.
Given a monomial 
\begin{displaymath}
m = (x_{i_1}\wedge x_{i_2} \ldots \wedge x_{i_r}) \otimes (e_{j_1} \otimes e_{j_2} \ldots \otimes e_{j_n}) \in P_n
\end{displaymath}
we define
\begin{displaymath}
\Psi(m) = (x_{i_1}x_{i_2} \ldots x_{i_r}) \otimes (e_{j_1} \otimes e_{j_2} \ldots \otimes e_{j_n}) \in Q_n.
\end{displaymath}
Clearly, the map $\Psi$ is a bijection. 
The map $\Psi$ also induces a bijection $\widetilde{\Psi}$ between the set of monomial submodules of $P_n$ and $Q_n$ in the following way: 
given a monomial submodule $M$ of $P_n$, the submodule $\widetilde{\Psi}(M) \subset Q_n$ is generated by the monomials $\Psi(m)$ with $m$ varying over all the monomials in $M$. 

\subsubsection*{Action of the increasing monoid} We define an action of $\Inc(\bN)$ on $P_n$. If $\sigma \in \Inc(\bN)$ and 
\begin{displaymath}
m = (x_{i_1} \wedge x_{i_2} \ldots \wedge x_{i_p}) \otimes (e_{j_1} \otimes e_{j_2} \otimes \ldots \otimes e_{j_n}) \in P_n,
\end{displaymath}
we define
\begin{displaymath}
\sigma m := (x_{\sigma(i_1)} \wedge x_{\sigma(i_2)} \ldots \wedge x_{\sigma(i_p)}) \otimes (e_{\sigma(j_1)} \otimes e_{\sigma(j_2)} \otimes \ldots \otimes e_{\sigma(j_n)})
\end{displaymath}
and linearly extend this to all of $P_n$. We similarly define an action of $\Inc(\bN)$ on $Q_n$ as well. 

The action we have just defined endows $P_n$ and $Q_n$ with the structure of an $\Inc(\bN)$-equivariant module over $R$ and $B$ respectively. 
Furthermore, the map $\widetilde{\Psi}$ restricts to an \textit{inclusion-preserving} bijection:
\begin{displaymath}
\xymatrix{
\{ \text{$\Inc(\bN)$-stable monomial submodules of $P_n$} \} \ar[r]^{\widetilde{\Psi}} & \{ \text{$\Inc(\bN)$-stable monomial submodules of $Q_n$} \} }
\end{displaymath}

\subsubsection*{Total order on monomials of $P_n$} Fix an $n \in \bN$. We endow a total ordering (denoted $\le$) on the set of monomials of $P_n$. Given monomials $m = (x_{i_1}\wedge x_{i_2} \ldots \wedge x_{i_r}) \otimes (e_{j_1} \otimes e_{j_2} \ldots \otimes e_{j_n})$, we associate to it the word $(I, J) = ((i_1, i_2, \ldots, i_r), (j_1, j_2, \ldots, j_n))$ in $\bN^* \times \bN^n$ {(here, $\bN^*$ is the set of words of arbitrary length on the alphabet $\bN$)}. The lexicographic order on {$\bN^* \times \bN^n$} restricts to a total order on the set of {monomials}. For example, in $P_3$, the monomial $(x_3 \wedge x_1) \otimes (e_1 \otimes e_4 \otimes e_1)$ will be bigger than $(x_2 \wedge x_1) \otimes (e_7 \otimes e_2 \otimes e_4)$ but smaller than $x_5 \otimes (e_1 \otimes e_1 \otimes e_1)$.

The total order we have defined is also compatible with multiplication by a monomial of $R$ in the following sense: for a monomial $n \in R$, and monomials $m, m' \in P_n$ with $m > m'$, we have $nm > nm'$. Here, we are implicitly identifying a monomial with its negative, as it is possible that $nm$ or $nm'$ is not a monomial as we have defined, in which case $-nm$ or $-nm'$ will be a monomial.

\subsubsection*{Initial submodules}Given an element $m$ in $P_n$, we can write it as a sum of monomials $m = \sum a_i m_i$, with $a_i \in k^\star$, and $m_i$ being monomials. We define the \textit{initial term} of $m$, denoted $\init(m)$, to be $a_n m_n$, where $m_n$ is the largest monomial (under the total order just defined) among all the $m_i$.
It is easy to see that $\init(\sigma(m)) = \sigma \init(m)$, for $\sigma \in \Inc(\bN)$ and $m \in P_n$.
Given an arbitrary $\Inc(\bN)$-stable submodule $M$ of $P_n$, we define its \textit{initial submodule}, denoted $\init(M)$, to be the submodule generated by all elements of the form $\init(m)$, with $m \in M$. By definition, the submodule $\init(M)$ is an $\Inc(\bN)$-stable monomial submodule of $P_n$.

\begin{theorem}\label{thm:noetherianity}
$\Mod_R$ is locally noetherian. 
\end{theorem}
\begin{proof}
For each $\lambda$, let $P_\lambda = R \otimes \Div^\lambda(\bV)$. Every finitely generated $R$-module is a quotient of a finite direct sum of $P_\lambda$ (as $\lambda$ varies). Furthermore, the module $P_\lambda$ is a submodule of $P_n$ for $n = |\lambda|$ (as $\Div^\lambda(\bV)$ is a subrepresentation of $\bV^{\otimes |\lambda|}$). Therefore, it suffices to show that $P_n$ is $\GL$-noetherian for all $n$, because finite direct sums, submodules and quotients of $\GL$-noetherian modules are $\GL$-noetherian. We will prove that $P_n$ is $\Inc(\bN)$-noetherian; this suffices as every $\GL$-stable submodule of $P_n$ is easily seen to be $\Inc(\bN)$-stable.

Fix an $n$. For $\Inc(\bN)$-stable submodules $M \subset N \subset P_n$, clearly $\init(M) \subset \init(N)$. The key point here is that the $\init$ operation detects strict inclusions, i.e., for $M \subset N \subset P_n$, we have $\init(M) = \init(N)$ if and only if $M = N$ (see proof of \cite[Proposition~4.2.2]{ss17grobner}). Therefore, if ACC holds for $\Inc(\bN)$-stable monomial submodules of $P_n$, then ACC holds for all $\Inc(\bN)$-stable submodule.

Using the bijection $\widetilde{\Psi}$, it suffices to prove that ACC holds for $\Inc(\bN)$-stable monomial submodules of $Q_n$, which follows from Theorem~\ref{thm:cohen}.
\end{proof}

\subsection{Torsion and generic \texorpdfstring{$R$}{R}-modules}\label{ss:torsion}
An element $m$ in an $R$-module is \textit{torsion} if $I m = 0$ for a nonzero $\GL$-ideal of $R$. Since the only nonzero proper $\GL$-ideals of $R$ are of the form $\fm^n$ for $n \ge 1$, an element $m$ is torsion if and only if $\fm^n m = 0$ for some $n \in \bN$. A module $M$ is \textit{torsion} if all elements of $M$ are torsion. 

The subcategory $\Mod_R^{\tors}$ of torsion modules is a Serre subcategory of $\Mod_R$. We let $\Mod_R^{\gen} = \Mod_R/\Mod_R^{\tors}$. A \textit{generic module} is an object of $\Mod_R^{\gen}$. All the functors from Section~\ref{ss:serre} exist in this case. In particular, the functor $\Gamma: \Mod_R \to \Mod_R^{\tors}$ takes a module $M$ to its submodule $\Gamma(M)$ of torsion elements. A nonzero module $M$ is \textit{torsion-free} if $\Gamma(M) = 0$.

We now give some examples of torsion and torsion-free $R$-modules. Given a nonzero module $M$, let $M^{>n}$ denote the submodule of $M$ containing all elements of degree $> n$. The module $M/M^{>n}$ is torsion (provided it is nonzero). Any finite length $R$-module is clearly torsion. The module $R$ is torsion-free: given a nonzero element $r \in R$, it is clear that $ \fm^n r  \ne 0$ for all $n \ge 0$ (since $\fm^n \fm^m = \fm^{n+m} \ne 0$). We have crucially used the fact that we have infinitely many variables here, for otherwise, $\fm^n = 0$ for $n \gg 0$. More generally, the induced module $R \otimes W$ are all torsion-free, and so are the semi-induced $R$-modules. 

\begin{lemma}\label{lem:torisfinlen}
The submodule $\Gamma(M)$ of a finitely generated $R$-module $M$ is finite length.
\end{lemma}
\begin{proof}
First, assume $M$ is a finitely generated torsion $R$-module generated by homogeneous elements $m_1, m_2, \ldots, m_n$. Assume $\fm^r m_i = 0$ for all $i$. Then $M_i$ is a finite length $\GL$-representation for all $i$ (this holds all finitely generated $R$-modules), and vanishes for all $i > r + \max\deg(m_i)$. So $M$ has finite length as a $\GL$-representation. If $M$ is a finitely generated $R$-module, then the submodule $\Gamma(M)$ is a finitely generated torsion $R$-module by Theorem~\ref{thm:noetherianity}, and hence $\Gamma(M)$ is finite length, as required.
\end{proof}
We now verify property (Inj) for $\Mod_R^{\tors}$, {i.e., injectives in the Serre subcategory $\Mod_R^{\tors}$ remain injective in the ambient category $\Mod_R$.}
\begin{lemma}\label{lemma:propinj}
    $\Mod_R^{\tors}$ satisfies property (Inj).
\end{lemma}
\begin{proof}
Let $M$ be a finitely generated $R$-module. The submodule $\Gamma(M)$ is supported in finitely many degrees by Lemma~\ref{lem:torisfinlen}; assume $\Gamma(M)_i = 0$ for $i > N$. Let $M^{>N}$ be the submodule of $M$ consisting of all elements of degrees $> N$. The submodule $M^{>N}$ is torsion-free and the quotient $M/M^{>N}$ is torsion. The result now follows from \cite[Proposition~4.18]{ss19gl2}.
\end{proof}
This implies:
\begin{corollary}\label{cor:lcvanishingtorsion}
    Let $M$ be a torsion $R$-module. For $i > 0$, the local cohomology functors $\rR^i\Gamma(M)$ vanish.
\end{corollary}
\begin{proof}
This is \cite[Proposition~4.2]{ss19gl2}.
\end{proof}
Even if $rm =0$ for an element $r \in \fm$ and $m \in M$, the element $m$ need not be torsion. For example, we have $x_1 \wedge x_1 = 0$, but $R$ is torsion-free. The next lemma shows that if an element $m$ is killed by $r \in R$ which is \textit{disjoint} from $m$ (i.e., the basis vectors of $\bV$ which are used in $m$ are disjoint from the ones used in $r$), then $m$ is a torsion element.
\begin{lemma}\label{lem:torsiontrick}
    Let $M$ be a finitely generated $R$-module, and $m \in M$ be an element such that $m \in M(k^n)$, i.e., $m$ only uses the first $n$ basis vectors. If $x_{n+1}x_{n+2}\ldots x_{n+r} m = 0$, then $m$ is a torsion element killed by $\fm^r$. 
\end{lemma}
\begin{proof}
We give a proof when $r = 1$, leaving the easy generalization to the reader. Given $i \ne n+1$, let $g$ be the element in $\GL$ that maps $e_{n+1}$ to $e_{n+1} + e_i$, and fixes $e_j$ for $j \ne n+1$. We have 
\begin{displaymath}
0 = g (x_{n+1} m) = (g x_{n+1}) (g m) = (x_i + x_{n+1}) m = x_i m + x_{n+1} m = x_im,
\end{displaymath}
where for the third equality, we use the fact that $m$ is fixed by $g$ since $m$ only uses the first $n$ basis vectors. So the $\GL$-ideal $\fm$ kills $m$, as required.
\end{proof}

{
For an abelian category $\mathcal{C}$, the \textit{Krull--Gabriel dimension} measures how far $\mathcal{C}$ is from being locally of finite length (see \cite[\S~4.1]{gab62serre} for a precise definition). We have $\kgdim(\mathcal{C}) = 0$ if and only if all objects in $\mathcal{C}$ are locally finite length. It is easy to see using Lemma~\ref{lem:torisfinlen} that $\kgdim(\Mod_R^{\tors}) = 0$. We prove that $\Mod_R^{\gen}$ is also zero-dimensional.
}
\begin{lemma}\label{lem:uniformboundonwedge}
    Fix an $n \in \bN$. There exists an $F(n) \in \bN$ such that the length of $\lw^i(\bV) \otimes \bV^{\otimes n}$ as a $\GL$-representation is less than $F(n)$ for all $i \ge 0$.
\end{lemma}
\begin{proof}
By \cite[Theorem~2.8]{har15per}, the lengths of $\lw^i(\bV) \otimes \bV^{\otimes n}$ is eventually periodic, i.e., there exists $q$ (a power of $p$) and $N$ such that for $r, s > N$ with $q | (r-s)$, we have
\begin{displaymath}
\len_{\GL}(\lw^r(\bV) \otimes \bV^{\otimes n}) = \len_{\GL}(\lw^s\bV \otimes \bV^{\otimes n}).
\end{displaymath}
In particular, the lengths are bounded above by some $F(n) \in \bN$, as required.
\end{proof}
\begin{proposition}\label{prop:kgdimone}
   The category $\Mod_R^{\gen}$ has Krull--Gabriel dimension zero. Consequently, the Krull--Gabriel dimension of $\Mod_R$ is one.
\end{proposition}
\begin{proof}
{We first show that for a finitely generated $R$-module $M$, we have 
\begin{displaymath}
    \len_{\Mod_R^{\gen}}(T(M)) < \liminf_i \len_{\GL}(M_i). 
\end{displaymath}
Indeed, given a submodule $N \subset M$, the equality $T(N) = T(M)$ holds if and only if $M/N$ is finite length, which holds if and only $N_i = M_i$ for $i \gg 0$. Therefore, if $T(N) \subsetneq T(M)$, then we have $\len_{\GL}(N_i) < \len_{\GL}(M_i)$ for sufficiently large $i$, whence we get the above inequality.}

Now, every finitely generated $R$-module is a subquotient of direct sums of $P_n = R\otimes \bV^{\otimes n}$. So it suffices to show that $T(P_n)$ has finite length. By Lemma~\ref{lem:uniformboundonwedge} and the discussion in the previous {paragraph}, the length of $T(P_n)$ is {at most} $F(n)$, as required.
\end{proof}

\section{The Shift Theorem}\label{s:shift}
In the last section, we showed that $\Mod_R$ splits into two pieces, the Serre subcategory of torsion $R$-modules, and the Serre quotient of generic $R$-modules.
In this section, we prove the shift theorem which, in some sense, lets us glue these two pieces together. We implicitly use the equivalence $\Phi$ between $\Pol$ and $\Rep^\pol(\GL)$ throughout. In particular, we suppress the functor $\Phi$. So if $M$ is a polynomial representation, and we write $M(W)$ for some vector space $W$, we mean the $\GL(W)$-representation $\Phi(M)(W)$.

\subsection{The shift functor}\label{ss:shiftdefn}
As we explained in Section~\ref{ss:schurder}, the Schur derivative restricts to an endofunctor of $\Mod_R$. We now define a natural map $i_M \colon M \to \Sh(M)$ for an $R$-module $M$. We have the $R$-module map $m \mapsto y_1m$ from $M$ to $M(k \oplus \bV)$ (where $y_1$ is the new variable in $R(k\oplus \bV)$). Since $m$ uses only the basis vectors $e_i$, the $k^\star \subset \GL(k \oplus \bV)$-weight of $y_1 m$ is $1$. Therefore, the map factors $M \to \Sh(M) \to M(k \oplus \bV)$. We let $i_M$ be the first map. It is easy to see that $i$ is natural. 
We let $\De(M) = \coker(i_M)$ be the \textit{difference functor}. Since $\Sh$ is an exact functor, it follows from the snake lemma that $\De$ is right exact. We also let $\Sh^r$ be the $r$-fold composition of $\Sh$ for $r > 0$ and let $\De_r(M) = \coker(M \to \Sh^r(M))$. 

We first prove some basic properties of the Schur derivative functor (which hereafter we refer to as the \textit{shift functor}). See \cite[Section~2]{ly17fi} for analogous results for $\FI$-modules.

\begin{proposition}\label{prop:basicshift}
    Let $M$ be an $R$-module.
    \begin{enumerate}[label=(\alph*)]
        \item The kernel of $i_M$ is torsion.
        \item The map $i_M$ is injective if and only if $M$ is torsion-free.
        \item If $M$ is torsion-free then so is $\Sh(M)$.
        \item If $M$ is finitely generated, then $\Sh^r(M)$ is torsion-free for sufficiently large $r$.
        \item {If $M$ is nonzero, the module $\De_r(M)$ is generated in degrees $\le t_0(M) - 1$ for all $r \geq 1$.}
        \item If $M$ is a finitely generated semi-induced module, then so are $\Sh(M)$ and $\De(M)$.
        \item If $M$ is finitely generated, then so are $\Sh(M)$ and $\De(M)$.
        \item {For all $r \geq 1$, the functors $\Sh\De_r$ and $\De_r\Sh$ are naturally isomorphic (i.e., the functors $\Sh$ and $\De_r$ commute).}
        \item The functors $\Sh\Gamma$ and $\Gamma\Sh$ are naturally isomorphic (i.e., the functors $\Sh$ and $\Gamma$ commute).
    \end{enumerate}
\end{proposition}

The proposition will be proved at the end of this subsection, after we prove a few preliminary results. 

\begin{lemma}\label{lem:obvshiftlemma}
    Let $i \colon \id \to \Sh$ be the natural map of functors. We have two natural maps from $\Sh$ to $\Sh^2$: one given by $\Sh (i)$, and the other given by $i_{\Sh}$. There exists an involution $\tau$ of $\Sh^2$ such that $\Sh(i)=\tau \circ i_{\Sh}$.
\end{lemma}
\begin{proof}
The module $\Sh^2(M)$ can be identified as a submodule of $M(k^2 \oplus \bV)$. Under this identification, the map $\Sh(i_M)$ is given by multiplying with the variable introduced first (call it $y_1$), and $i_{\Sh(M)}$ is given by multiplying with the variable introduced second (call it $y_2$). We have an isomorphism of $M(k^2 \oplus \bV)$ induced by the isomorphism of $k^2 \oplus \bV$ which swaps the two basis vectors of the first factor and is identity on $\bV$. This {isomorphism} clearly restricts to a natural isomorphism of $\Sh^2(M)$; call this natural map $\tau$. We have 
\begin{displaymath}
(\tau_M \circ \Sh(i_M))(m) = \tau_M(y_1m) = y_2 m = i_{\Sh(M)}(m),
\end{displaymath}
as required.
\end{proof}

\begin{lemma} \label{lem:shiftofinduced}
    Let $V$ be a polynomial representation of $\GL$. We have isomorphisms
   \begin{displaymath}
   \Sh(R \otimes V) \cong (\Sh(R) \otimes V) \oplus (R \otimes \Sh(V)) \cong (R \otimes V) \oplus (R \otimes \Sh(V)).
   \end{displaymath}
   Furthermore, the natural map $i_{R \otimes V}$ is the identity map onto the first summand. Therefore, the module $\De(R\otimes V)$ is isomorphic to $R \otimes \Sh(V)$.
\end{lemma}
\begin{proof}
Firstly, it is easy to check that $i_R$ is an isomorphism (see Lemma~\ref{lem:shiftofsym}). Therefore $R \cong \Sh(R)$, giving us the second isomorphism in the statement of the lemma. For the first isomorphism, we use the fact that $\Sh$ is a categorical derivation. The claim about $i_{R \otimes V}$ follows from the definition of $i$. 

\end{proof}
\begin{corollary}\label{cor:dfrobzero}
    Let $W$ be a polynomial representation of $\GL$. Then $\De(R\otimes W^{(1)}) = 0$.
\end{corollary}
\begin{proof}
All the weights of $W^{(1)}$ are divisible by $p$, so $\Sh(W^{(1)}) = 0$. The result now follows from Lemma~\ref{lem:shiftofinduced}.
\end{proof}
\begin{proof}[Proof of Proposition~\ref{prop:basicshift}]
\hfill
\begin{enumerate}[label=(\alph*)]
\item Assume $m$ is in the kernel of $i_M$, i.e., $y_1 m = 0$ in $\Sh(M)$. Viewing $\Sh(M)$ as a submodule of $M(k \oplus \bV)$, we see that the element $m \in M(\bV) \subset M(k \oplus \bV)$ is annihilated by $y_1 \in R(k \oplus \bV)$. Since $y_1$ and $m$ are disjoint, the element $m$ is torsion by Lemma~\ref{lem:torsiontrick}.
\item If $M$ is torsion-free, then by part (a), the map $i_M$ is injective. Now, suppose $M$ is not torsion-free; let $m$ be a torsion element and assume $m \in M(k^i)$ (i.e., it only uses the basis vectors $e_1, e_2, \ldots, e_i$). Let $n$ be the minimal integer for which $\fm^n m = 0$. Pick a nonzero element $\widetilde{m}$ in the submodule $\fm^{n-1} m$. Since $\fm^n m=0$, we have $\fm\widetilde{m} = 0$. So, in $M(k \oplus \bV)$, we also have that $y_1 \widetilde{m} = 0$, which implies that $\widetilde{m}$ is in the kernel of $i_M$, as required. 
\item If $M$ is torsion-free, then by part (a), the map $i_M$ is injective. Applying the exact functor $\Sh$, we see that the map $\Sh(i_M) \colon \Sh(M) \to \Sh^2(M)$ is also injective. By Lemma~\ref{lem:obvshiftlemma}, we have $i_{\Sh(M)} = \tau_M \circ \Sh(i_M)$, where $\tau_M$ is an involution of $\Sh^2(M)$. Therefore, the map $i_{\Sh(M)}$ is also injective, and so by part (b), the module $\Sh(M)$ is torsion-free.
\item The submodule $\Gamma(M)$ of torsion elements is finitely generated, and so supported in degrees $< n$ for some $n \in \bN$. We have the short exact sequence 
\begin{displaymath}
0 \to \Gamma(M) \to M \to M/\Gamma(M) \to 0. 
\end{displaymath}
Applying $\Sh^n$, we obtain an isomorphism $\Sh^n(M) \cong \Sh^n(M/\Gamma(M))$ as $\Sh^n(\Gamma(M)) = 0$. The module $\Sh^n(M/\Gamma(M))$ is torsion-free by part (c), and so $\Sh^n(M)$ is also torsion-free, as required.
\item {We prove the $r=1$ case which can be easily extended to all $r$ by the reader.} For a polynomial representation $W$ of degree $\le n$, the representation $\Sh(W)$ has degree $\le n-1$. 
Therefore, by Lemma~\ref{lem:shiftofinduced}, applying $\De$ to an induced module generated in degree $\le n$ results in an induced module generated in degree $\le n-1$. Now, for an $R$-module $M$, we have a surjection $R\otimes W \to M$, with $W$ a polynomial representation of degree $\le t_0(M)$. Since $\De$ is right exact, we have a surjection $R \otimes \De(W) \to \De(M)$, which implies that $t_0(\De(M)) \le t_0(M) - 1$, as required. 
\item This follows from Lemma~\ref{lem:shiftofinduced}, and induction on the length of the filtration of $M$.
\item This follows from Lemma~\ref{lem:shiftofinduced} and the fact that the Schur derivative of a finite length polynomial representation is also finite length. 
\item {We again prove the $r=1$ case with the easy generalization for all $r > 1$ left to the reader}. The module $\Sh(\De(M))$ is the cokernel of $\Sh(i_M)$, and $\De(\Sh(M))$ is the cokernel of $i_{\Sh(M)}$. These two maps only differ by an automorphism of the target, and so their cokernels are isomorphic.
\item We have the short exact sequence $0 \to \Gamma(M) \to M \to M/\Gamma(M) \to 0$. Applying the functor $\Sh$ to this short exact sequence, we get 
\begin{displaymath}
0 \to \Sh(\Gamma(M)) \to \Sh(M) \to \Sh(M/\Gamma(M)) \to 0.
\end{displaymath}
By part (c), the module $\Sh(M/\Gamma(M))$ is torsion-free, and therefore, the submodule $\Gamma(\Sh(M))$ is contained in the image of $\Sh(\Gamma(M))$. For the reverse containment, the image of $\Sh(\Gamma(M))$ is a torsion submodule of $\Sh(M)$, and so is contained in $\Gamma(\Sh(M))$.
\end{enumerate}
\end{proof}

\begin{lemma}\label{lem:handlingtorsion}
{   Let $0 \to L \to M \to N \to 0$ be a short exact sequence of $R$-modules. We have a six-term exact sequence
   \begin{displaymath}
       0 \to \bK(L) \to \bK(M) \to \bK(N) \to \De(L) \to \De(M) \to \De(N) \to 0
   \end{displaymath}
   where $\bK = \ker(\id \to \Sh)$.}
\end{lemma}
\begin{proof}
{
    This is just the snake lemma applied to the diagram}
\begin{displaymath}
\begin{tikzcd}
  0 \arrow[r] & L \arrow[d, "i_L"] \arrow[r] & M \arrow[d, "i_M"] \arrow[r] & N\arrow[d, "i_N"] \arrow[r] & 0 \\
  0 \arrow[r] & \Sh(L) \arrow[r] & \Sh(M) \arrow[r] & \Sh(N) \ar[r] & 0.
\end{tikzcd}
\end{displaymath}
\end{proof}
\begin{remark}
{It is not hard to see that the functor $\bK$ is isomorphic to $\bL^1 \De$, and $\bL^i \De$ vanishes for $i > 1$ (for $\FI$-modules, the analogous result can be found in \cite[Lemma~4.7]{ce17regularity}).}
\end{remark}
\begin{corollary}\label{cor:handlingtorsion}
Let $0 \to L \to M \to N \to 0$ be a short exact sequence of $R$-modules such that $N$ is torsion-free. We have a short exact sequence 
\begin{displaymath}
0 \to \De(L) \to \De(M) \to \De(N) \to 0
\end{displaymath}
\end{corollary}
\begin{proof}
The module $\bK(N) = 0$ by Proposition~\ref{prop:basicshift}(a). The result now follows from {Lemma}~\ref{lem:handlingtorsion}.
\end{proof}
\subsection{Vanishing of \texorpdfstring{$\De$}{Delta}} \label{ss:deltavanishing}
This subsection contains the main new technical result of this paper. We classify the torsion-free $R$-modules for which $\De(M)$ vanishes. In characteristic zero, it is not hard to prove that for a nonzero $R$-module $M$, we have $t_0(\De(M)) = t_0(M) - 1$ \cite[Proposition~2.4]{ly17fi}. So if $\De(M) = 0$, then $M$ must be generated in degree zero and therefore induced. In positive characteristic we only have the inequality $t_0(\De(M)) \le t_0(M) - 1$. In this subsection, we prove that if $\De$ vanishes for a torsion-free module, then it is an extension of modules of the form $R \otimes W^{(1)}$. In particular, they are all semi-induced. We first prove this when the module $M$ is generated in one degree (Proposition~\ref{prop:dmzerogenonedeg}), and then use the short exact sequence $0 \to M^{<n} \to M \to M/M^{<n} \to 0 $ to deduce the result in general (Proposition~\ref{prop:dmzeroisflat}). 

The key lemma is:
\begin{lemma}\label{lem:weightsinfrob}
    Let $W$ be a polynomial representation of $\GL$ that is not Frobenius twisted. There exists a weight vector $w \in W$ of weight $\mu$ with $\mu_1 = 1$.
\end{lemma}
\begin{proof}
    Since $W$ is not Frobenius twisted, at least one of the irreducible constituent of $W$ will not be Frobenius twisted. So we may assume that $W$ is an \textit{irreducible} polynomial representation that is not Frobenius twisted. By the Steinberg tensor product theorem, we have an isomorphism $W \cong L_{\lambda^0} \otimes L_{\lambda^1}^{(1)} \otimes L_{\lambda^2}^{(2)} \ldots \otimes L_{\lambda^r}^{(r)}$ for $p$-restricted partitions $\lambda^0, \lambda^1, \ldots, \lambda^r$ with $\lambda^0 \ne \emptyset$. 
    In $L_{\lambda^0}$, we can find a nonzero weight vector $v_0$ of weight $\alpha$ with $\alpha_1 = 1$ and we can find a weight vector $v_1$ of weight $\beta$ with $\beta_1 = 0$ in $L_{\lambda^1}^{(1)} \otimes L_{\lambda^2}^{(2)} \ldots \otimes L_{\lambda^r}^{(r)}$. The vector $ v_0 \otimes v_1$ is a weight vector in $W$, and the first component of its weight is $1$, as claimed. 
\end{proof}
\begin{lemma}\label{lem:dmzero}
    {Assume $M$ is a finitely generated $R$-module with $M_n = 0$ and $M_{n+1}$ not Frobenius twisted. Then $\De(M)_n \ne 0$. In particular, if $M$ is a finitely generated $R$-module generated in degree $n+1$ with $\De(M) = 0$, then $M_{n+1}$ is Frobenius twisted.}
\end{lemma}
\begin{proof}
    {Since the representation $M_{n+1}$ is not Frobenius twisted, by Lemma~\ref{lem:weightsinfrob}, we can find a vector $v$ of weight $(1, \lambda)$ in $M_n$. In $M(k \oplus \bV)$, let $w = g \cdot v$ where $g$ swaps the newly introduced basis vector, say $f$, with $e_1$ (and fixes the other $e_i$). The vector $w$ lies in $\Sh(M)$ as it has weight $(1, 0, \lambda)$ for the action of {$\GL(k \oplus \bV)$}. Furthermore, the element $w$ has degree $n$ so it is not in the image of the natural map $M \to \Sh(M)$ (as $M_n = 0$), which implies that the image of $w$ is nonzero in $\De(M)$, as required. The second part of the statement follows easily since for a module generated in degree $n+1$, we have $M_n = 0$.}
\end{proof}
\begin{proposition}\label{prop:dmzerogenonedeg}
    Let $M$ be a torsion-free $R$-module generated in degree $n$ such that $\De(M) = 0$. The natural map $R \otimes M_n \to M$ is an isomorphism, i.e., $M$ is an induced module. 
\end{proposition}
\begin{proof}
    The irreducible components of $R$ in positive degree are $p$-restricted representations of $\GL$ (in degree $i$, it is the $i$-th exterior power $\lw^i(\bV)$ which is irreducible with highest weight $(1^i)$). Therefore, if $W$ is an irreducible representation of $\GL$, then $R_n \otimes W^{(1)}$ is also irreducible by the Steinberg tensor product theorem (Theorem~\ref{thm:steinberg}), and in particular, if $n>0$, it is not Frobenius twisted. It follows that for a finite length polynomial representation $W$, the irreducible components of $R \otimes W^{(1)}$ are not Frobenius twisted in degrees $ > \deg(W^{(1)})$.
    
    By the assumptions on $M$, we see that $M_n$ is Frobenius twisted by Lemma~\ref{lem:dmzero}. Consider the natural map $\phi \colon R \otimes M_n \to M$. The map $\phi$ is surjective, and $K = \ker(\phi)$ is zero in degrees $\le n$. We have to prove that $K = 0$. First, we claim that $\De(K) = 0$. Indeed, since $M$ is torsion-free, by Corollary~\ref{cor:handlingtorsion}, we have a short exact sequence 
     \begin{displaymath}
    0 \to \De(K) \to \De(R \otimes M_n) \to \De(M) \to 0.
    \end{displaymath} 
    By Corollary~\ref{cor:dfrobzero}, we have $\De(R \otimes M_n) = 0$ and so $\De(K) = 0$, as claimed. Now, suppose $K$ is nonzero, and let $r$ be the minimal degree in which $K$ is nonzero (i.e., $K_r \ne 0$ and $K_i = 0$ for $i < r$). Note $r > n$ and so $K_r$ is not Frobenius twisted by the first paragraph. {Therefore, by Lemma~\ref{lem:dmzero}, we see that $\De(K)$ must be nonzero, which is a contradiction. So the map $\phi$ is an isomorphism, as required.}
\end{proof}

\begin{lemma}\label{lem:torsionfreequotient}
    Let $M$ be a torsion-free $R$-module with $t_0(M) = n$ such that $\De(M)$ is semi-induced and $\De(M/M^{<n}) = 0$. Then $M/M^{<n}$ is torsion-free.
\end{lemma}
\begin{proof}
    As in Lemma~\ref{lem:handlingtorsion}, we have the short exact sequence 
    \begin{displaymath}
    0 \to \bK(M/M^{<n}) \to \De(M^{<n}) \to \De(M) \to 0.
    \end{displaymath} 
    Applying $\Tor$ to this gives us the exact sequence
    \begin{displaymath}
    \Tor_1^R(\De(M), k) \to \Tor_0^R(\bK(M/M^{<n}),k) \to \Tor_0^R(M^{<n},k).
    \end{displaymath}
    However, since $\De(M)$ is semi-induced, the module $\Tor_1^R(\De(M), k) = 0$, which implies that 
    \begin{displaymath}
    t_0(\bK(M/M^{<n})) \leq t_0(M^{<n}) < n.
    \end{displaymath}
    But $\bK(M/M^{<n})$ is supported only in degrees $\ge n$ (as it is a submodule of $M/M^{<n}$), and so must be zero. Therefore, the natural map $M/M^{<n} \to \Sh(M/M^{<n})$ is injective, which implies that the module $M/M^{<n}$ is torsion-free by Proposition~\ref{prop:basicshift}(b).
\end{proof}
We can now prove our main result.
\begin{proposition}\label{prop:dmzeroisflat}
Let $M$ be a finitely generated torsion-free $R$-module such that $\De(M) = 0$. Then $M$ is semi-induced.
\end{proposition}
\begin{proof}
    We proceed by induction on the generation degree of $M$.
    When $t_0(M) = 0$, this follows from Proposition~\ref{prop:dmzerogenonedeg}.
    Now, assume that $t_0(M) = n > 0$. Since $\De$ is right exact, we see that $\De(M/M^{<n}) = 0$. Therefore, by Lemma~\ref{lem:torsionfreequotient}, the module $M/M^{<n}$ is torsion-free, and so by Proposition~\ref{prop:dmzerogenonedeg}, the module $M/M^{<n}$ is semi-induced. By Corollary~\ref{cor:handlingtorsion}, we have a short exact sequence 
    \begin{displaymath}
    0 \to \De(M^{<n}) \to \De(M) \to \De(M/M^{<n}) \to 0.
    \end{displaymath}
    As $\De(M)=0$, we have that $\De(M^{<n}) = 0$, which implies that $M^{<n}$ is semi-induced (by induction) and therefore, so is $M$. 
\end{proof}
We finally note a corollary {(of Lemma~\ref{lem:torsionfreequotient})} that we will use in the proof of the shift theorem.
\begin{corollary}
 \label{cor:dqlowdeg}
     Let $M$ be a finitely generated torsion-free $R$-module with $t_0(M) = n$ such that $\De(M)$ is semi-induced with $t_0(\De(M)) \le \max(-1, n - 2)$. Then $M/M^{<n}$ is semi-induced. 
 \end{corollary}
 \begin{proof}
     Since the functor $\De$ is right exact, we have \begin{displaymath}
     t_0(\De(M/M^{<n})) \le t_0(\De(M)) < n-1.
     \end{displaymath}
     Since $M/M^{<n}$ is supported in degrees $\ge n$, the module $\Sh(M/M^{<n})$ is supported in degrees $\ge n-1$, and so $\De(M/M^{<n})$ is also supported in degrees $\geq n-1$. Therefore, the above inequality implies that $\De(M/M^{<n}) = 0$. Now, by Lemma~\ref{lem:torsionfreequotient}, the module $M/M^{<n}$ is torsion-free and so by Proposition~\ref{prop:dmzerogenonedeg}, the module $M/M^{<n}$ is semi-induced.
 \end{proof}

\subsection{Semi-induced subquotients of semi-induced modules}\label{ss:subofsemi}
We prove some lemmas bounding the generation degree of semi-induced subquotients of semi-induced modules. The only idea used in this subsection is that $R(k^n)$, the exterior algebra on $k^n$, is supported in degrees $0$ through $n$.

\begin{lemma}
	Let $M$ be an $R$-module. Then for all $n$, we have $M(k^n)_i = 0$ for $i > n + t_0(M)$. Furthermore, if $M$ is semi-induced, then $M(k^n)_{n + t_0(M)}$ is nonzero for all sufficiently large $n$.
\end{lemma}
\begin{proof}
    This is clear since $R$ evaluated at $k^n$ is supported in degrees $\le n$.
\end{proof}

\begin{corollary}\label{corollary:subsemi}
Let $F$ be a semi-induced module and $Z$ be a submodule of $F$ such that $Z/Z^{< t_0(Z)}$ is semi-induced. Then $t_0(Z) \leq t_0(F)$.
\end{corollary}
\begin{proof}
    For sufficiently large $n$, the module $Z(k^n)$ is nonzero in degree $n + t_0(Z)$ as this is true for $Z/Z^{< t_0(Z)}$ by the previous lemma. Since $Z$ is a submodule of $F$, we have $n + t_0(Z) \leq n + t_0(F)$ for large $n$ by the previous lemma, giving us the required inequality.
\end{proof}

\subsection{Proof of the shift theorem}\label{ss:pfshift}
\begin{lemma}\label{lem:t1lesst0}
    Let $M$ be a torsion-free $R$-module such that $\De(M)$ is semi-induced. Then $t_1(M) \leq t_0(M)$.
\end{lemma}
\begin{proof}
     Let $F$ be a semi-induced module surjecting onto $M$ such that $t_0(F) = t_0(M)$. We have an exact sequence 
     \begin{displaymath}
     0 \to Z \to F \to M \to 0. 
     \end{displaymath}
     The long exact sequence of $\Tor_{\bullet}^R(-, k)$ yields:
    \begin{displaymath}
    0 \to \Tor_1^R(M, k) \to \Tor_0^R(Z, k) \to \Tor_0^R(F, k) \to \Tor_0^R(M, k) \to 0. 
    \end{displaymath}
    Therefore, we have the inequality $t_1(M) \le t_0(Z)$. So it suffices to show that $t_0(Z) \le t_0(M)$.
    
    Since $M$ is torsion-free, we also obtain a short exact sequence 
    \begin{displaymath}
    0 \to \De(Z) \to \De(F) \to \De(M) \to 0
    \end{displaymath}
    by Corollary~\ref{cor:handlingtorsion}. The module $\De(F)$ is semi-induced by Proposition~\ref{prop:basicshift}(f), and $\De(M)$ is semi-induced by assumption. So by Corollary~\ref{cor:semises}, the module $\De(Z)$ is also semi-induced.
    
    We know $t_0(\De(Z)) \le t_0(Z) - 1$ by Proposition~\ref{prop:basicshift}(e). If $t_0(\De(Z)) < t_0(Z) - 1$, then $Z/Z^{< t_0(Z)}$ is semi-induced by Corollary~\ref{cor:dqlowdeg} (note that $Z$ is torsion-free being a submodule of $F$) and therefore, by Corollary~\ref{corollary:subsemi}, we have $t_0(Z) \leq t_0(F) = t_0(M)$. 
    
    Instead, if $t_0(\De Z) = t_0(Z) - 1$, then from the short exact sequence above, we have the inequality 
    \begin{displaymath}
    t_0(\De(Z)) = t_0(Z) - 1 \le t_0(\De(F)) \le t_0(F) - 1 = t_0(M) - 1,
    \end{displaymath}
    or $t_0(Z) \le t_0(M)$. 
\end{proof}

\begin{proposition}\label{prop:liftingsemi}
	Let $M$ be a finitely generated torsion-free $R$-module such that $\De(M)$ is semi-induced. Then $M$ is also semi-induced.
\end{proposition}
\begin{proof}
	We induct on $t_0(M)$. When $t_0(M) = 0$, the module $\De(M) = 0$, and so $M$ is semi-induced by Proposition~\ref{prop:dmzeroisflat}. Now, assume $t_0(M) = n > 0$. 
	We have the short exact sequence
	\begin{displaymath}
	0 \to M^{<n} \to M \to M/M^{<n} \to 0.
	\end{displaymath}
	So it suffices to prove that $M^{<n}$ and $M/M^{<n}$ are semi-induced.
	
	We first show that $M/M^{<n}$ is semi-induced. The long exact sequence of $\Tor_{\bullet}^R(-, k)$ associated to the above short exact sequence gives us 
	\begin{displaymath}
    \Tor_1^R(M, k) \to \Tor_1^R(M/M^{<n}, k) \to \Tor_0^R(M^{<n}, k)
    \end{displaymath}
	from which we see that $t_1(M/M^{<n}) \leq \max(t_1(M), t_0(M^{<n})) = \max(t_1(M), n-1) \leq n$. For the last step, we use Lemma~\ref{lem:t1lesst0} to get $t_1(M) \leq t_0(M) = n$. So by Lemma~\ref{lem:relationsinlowdeg}, the module $M/M^{<n}$ is semi-induced.
	
    We proceed to show that $M^{<n}$ is also semi-induced. Since $M/M^{<n}$ is semi-induced, it is torsion-free, and therefore, we have a short exact sequence
    \begin{displaymath}
    0 \to \De(M^{<n}) \to \De(M) \to \De(M/M^{<n}) \to 0
    \end{displaymath}
    by Corollary~\ref{cor:handlingtorsion}. By Proposition~\ref{prop:basicshift}(f), we have that $\De(M/M^{<n})$ is semi-induced, and so from the short exact sequence above, we see that $\De(M^{<n})$ is also semi-induced by Corollary~\ref{cor:semises}. Since $t_0(M^{<n}) < n$, by the induction hypothesis, it now follows that $M^{<n}$ is semi-induced, as required. 
\end{proof}

\begin{theorem}[Shift Theorem]
Let $M$ be a finitely generated $R$-module. We can find an $L \in \bN$ such that for all $l > L$, the module $\Sh^l(M)$ is semi-induced.
\end{theorem}
\begin{proof}
We follow the proof of \cite[Theorem~3.13]{ly17fi} and proceed by induction on $t_0(M)$. It suffices to prove the result for $\Sh^r(M)$ for $r \gg 0$, and so we may additionally assume $M$ is torsion-free by Proposition~\ref{prop:basicshift}(d). When $t_0(M) = 0$, the module $\De(M)=0$ and so $M$ is semi-induced by Proposition~\ref{prop:liftingsemi}. Now, assume $t_0(M) = n > 0$. We have a short exact sequence $0 \to M \to \Sh(M) \to \De(M) \to 0$ with $t_0(\De(M)) < n$. By induction, for $l\gg 0$ the module $\Sh^{l}(\De(M))$ is semi-induced. By Proposition~\ref{prop:basicshift}{(h)} we have $\Sh^{l}(\De(M)) \cong \De(\Sh^{l}(M))$, and so by Proposition~\ref{prop:liftingsemi} we have that $\Sh^l(M)$ is semi-induced, as required. 
\end{proof}

\section{Structure of \texorpdfstring{$\Mod_R^{\fgen}$}{fg-Rmod}}\label{s:fgrmod}
In this section, we use the shift theorem to study the structure of finitely generated $R$-modules.
\subsection{Generators of the derived category}\label{ss:gendbmodr}
In this subsection, we show that the derived category is generated by the torsion and flat $R$-modules. Using this, we prove that all finitely generated $R$-modules have finite regularity. We first note a useful corollary of the shift theorem. {Throughout this section, we let $\rD^b_{\fgen}(\Mod_R)$ denote the bounded derived category of $\Mod_R$ with finitely generated cohomology.}

\begin{theorem}[Resolution Theorem]\label{thm:resthm}
Let $M$ be a finitely generated $R$-module. We have a chain complex of $R$-modules
\begin{displaymath}
0 \to M \to P^0 \to P^1 \to \ldots \to P^r \to 0
\end{displaymath}
satisfying the following properties:
\begin{itemize}
\item each $P^i$ is a finitely generated semi-induced module with $t_0(P^i) \le t_0(M) - i$,
\item $r \leq t_0(M)$, and
\item the cohomology of this complex is torsion (and so supported in finitely many degrees).
\end{itemize}
{Furthermore, given a map of $R$-modules $f \colon M \to N$, we can find complexes $M \to P^{\bullet}$ and $N \to Q^{\bullet}$ satisfying the above properties, and a map of complexes $\tilde{f}$ extending $f$.}
\end{theorem}
\begin{proof}
We prove the result by induction on $n = t_0(M)$. If $n = 0$, then by the shift theorem, for $l \gg0$, the module $\Sh^l(M)$ is semi-induced. Set $P^0 = \Sh^l(M)$. The kernel of the natural map $M \to P^0$ is torsion by Proposition~\ref{prop:basicshift}(a), and this map must be surjective as the cokernel is generated in degrees $< t_0(M) = 0$ by Proposition~\ref{prop:basicshift}{(e)}. Therefore, the complex $0 \to M \to P^0 \to 0$ is the required complex for $M$. Now assume $n > 0$. Again, by the shift theorem, the module $\Sh^l(M)$ is semi-induced for $l \gg 0$ (and $t_0(\Sh^l(M)) = t_0(M)$). Set {$P^0 = \Sh^l(M)$}. As before, the kernel of the natural map from $M \to P^0$ is torsion. Furthermore, the cokernel of this map is generated in degrees $\le n-1$ by Proposition~\ref{prop:basicshift}{(e)}. By induction, we have a complex $0 \to N \to P^1 \to P^2 \ldots P^n \to 0$ satisfying all the conditions of the theorem for $N$. It is easy to check that the complex $0 \to M \to P^0 \to P^1 \to \ldots \to P^n \to 0$ satisfies all the desired properties. {The last statement is clear upon noting that if $M \to P^{\bullet}$ satisfies all the required properties, then so does the complex $M \to \Sh(P^{\bullet})$.}
\end{proof}

\begin{lemma}\label{lem:dbfgtorsion}
    Given a bounded complex $C$ of finitely generated $R$-modules with finite length cohomology, there exists a complex of finite length $R$-modules $D$ that is quasi-isomorphic to $C$. 
\end{lemma}
\begin{proof}
{Given an $R$-module $M$, we let $M^{>n}$ be the submodule of $M$ consisting of all elements of degree $> n$, and given a complex of $R$-modules $C$, we let $C^{>n}$ be the canonical subcomplex of $C$ with $(C^{>n})^i = (C^i)^{>n}$. Now, assume $C$ is a bounded complex of finitely generated $R$-modules and its cohomology is supported in degrees $< N$. We have a short exact sequence
\begin{displaymath}
    0 \to C^{>N} \to C \to C/C^{>N} \to 0
\end{displaymath}
with $C/C^{>N}$ a bounded complex of finite length modules. By assumption on the cohomology of $C$, the complex $C^{>N}$ is acyclic so the the complex $C$ is quasi-isomorphic to $C/C^{>N}$, as required.}
\end{proof}

\begin{proposition}\label{prop:triangle}
Given a finitely generated $R$-module $M$, there exists a distinguished triangle $T \to M \to F \to$ in $\rD^b_{\fgen}(\Mod_R)$, with $T$ quasi-isomorphic to a bounded complex of finite length $R$-modules, and $F$ quasi-isomorphic to a bounded complex of finitely generated semi-induced $R$-modules.
\end{proposition}
\begin{proof}
Let $M \to P^{\bullet}$ be a complex satisfying Theorem~\ref{thm:resthm}. We consider $M$ to be a complex supported in degree $0$, and let $j$ denote the map of complexes $M \to P^{\bullet}$. We obtain a distinguished triangle 
\begin{displaymath}
M \to P^{\bullet} \to \cone(j) \to
\end{displaymath}
in $\rD^b_{\fgen}(\Mod_R)$. The complex $\cone(j)$ has finite length cohomology, so is quasi-isomorphic to a finite complex of finite length torsion $R$-modules $T$ by Lemma~\ref{lem:dbfgtorsion}. The rotated triangle $T[-1] \to M \to P^{\bullet} \to$ satisfies the requirements.
\end{proof}
\begin{theorem}\label{thm:gendbmodr}
    The category $\rD^b_{\fgen}(\Mod_R)$ is generated as a triangulated category by the modules $L_\lambda$ and $R \otimes L_\lambda$ for arbitrary $\lambda$.
\end{theorem}
\begin{proof}
By Proposition~\ref{prop:triangle}, the smallest triangulated subcategory of $\rD^b_{\fgen}(\Mod_R)$ containing the torsion and semi-induced $R$-modules is all of $\rD^b_{\fgen}(\Mod_R)$. Any torsion $R$-module (resp.~semi-induced $R$-module) has a filtration where the successive quotients are $L_{\lambda}$ (resp.~$R\otimes L_{\lambda}$). So the modules $L_{\lambda}$ and $R\otimes L_{\lambda}$ generate $\rD^b_{\fgen}(\Mod_R)$, as required.
\end{proof}
We can now prove the theorem on finite regularity.
\begin{defn}
For an $R$-module $M$, we define the \textit{Castelnuovo--Mumford regularity} of $M$ to be the minimal integer $\rho$ such that $t_i(M) \le \rho + i$ for all $i$, or $\infty$ if no such integer exists.
\end{defn}
\begin{theorem}\label{thm:finreg2}
The Castelnuovo--Mumford regularity of a finitely generated $R$-module is finite.
\end{theorem}
\begin{proof}
Let $\lambda$ be a partition. The induced module $R \otimes L_\lambda$ is flat, and so has regularity $|\lambda|$. The module $k \cong R/\fm$ has regularity $0$ as witnessed by the Koszul resolution (which is a flat resolution of $k$). Applying the functor $-\otimes_k L_\lambda$ to the Koszul resolution of $k$, we obtain a flat resolution of $L_\lambda$. Therefore, the module $L_\lambda$ has regularity $|\lambda|$. Since these two classes of modules generate the bounded derived category, every bounded complex has finite regularity, and in particular, so does every finitely generated module.
\end{proof}
\begin{remark}\label{rmk:cebounds}
Using Nagpal's shift theorem \cite{nag19vi}, Gan--Li showed that the regularity of a finitely generated $\VI$-module $M$ is at most $t_0(M) + t_1(M)$ \cite[Theorem~3.2]{gl20vi}. The results (and proofs) of Section~2 and Theorem~3.2 of loc.~cit.~also hold for $R$-modules. The only additional ingredient required for their proof is that for a finite length $R$-module $M$, we have $\reg(M) \le \max\deg(M)$. This can be proved using the Koszul complex. Therefore, by mimicking their approach, we obtain the bound $\reg(M) \le t_0(M) + t_1(M)$ for every finitely generated $R$-module.
\end{remark}

\subsection{The shift functor commutes with local cohomology}\label{ss:shiftlccommute}
In Proposition~\ref{prop:basicshift}(i), we showed that the shift functor commutes with the functor $\Gamma$. In this subsection, we show that the shift functor also commutes with the derived functors of $\Gamma$. This technical result will be used to prove several finiteness results for local cohomology. We first explain the setup.

Let $M$ be an $R$-module, and let $M \to I^\bullet$ and $\Sh(M) \to J^\bullet$ be injective resolution of $M$ and $\Sh(M)$ respectively. The identity map of $\Sh(M)$ induces a map of complexes $\Sh(I^\bullet) \to J^\bullet$ as $J^\bullet$ is an injective resolution. Applying $\Gamma$ and taking cohomology, we obtain a map $ H^i(\Gamma(\Sh(I^\bullet))) \to H^i(\Gamma(J^\bullet)) = \rR^i\Gamma(\Sh(M))$. We also have natural isomorphisms,
\begin{displaymath}
H^i(\Gamma(\Sh(I^\bullet))) \cong H^i(\Sh(\Gamma(I^\bullet))) \cong \Sh(H^i(\Gamma(I^\bullet))) = \Sh(\rR^i\Gamma(M)),
\end{displaymath}
where for the first isomorphism, we use that $\Sh$ commutes with $\Gamma$ by Proposition~\ref{prop:basicshift}(i), and for the second isomorphism, we use that $\Sh$ is an exact functor, and therefore commutes with taking cohomology.
So we get a natural map $F_{i,M} \colon \Sh(\rR^i\Gamma(M)) \to \rR^i\Gamma(\Sh(M))$.

\begin{proposition}\label{prop:shcommutesderivedgamma}
The map $F_i$ is an isomorphism for all $i \ge 0$.
\end{proposition}
\begin{proof}
Our proof is an expanded version of the proof of \cite[Lemma~4.20]{nag19vi} and \cite[Proposition~A.3]{dj16lc}. We proceed by induction on $i$. For $i=0$, this {is just} Proposition~\ref{prop:basicshift}(h). Now, assume $i > 0$, and we know that $F_{i-1}$ is an isomorphism. We have to show that $F_i$ is an isomorphism.

Before we prove the result, we make two reductions. First, we may assume that $M$ is torsion-free. Indeed, given an arbitrary module $N$, we have the short exact sequence $0 \to \Gamma(N) \to N \to N/\Gamma(N) \to 0$. Using the long exact sequence of local cohomology and Corollary~\ref{cor:lcvanishingtorsion}, we see that for $i > 0$, the map $F_{i, N}$ is an isomorphism if and only if $F_{i, N/\Gamma(N)}$ is an isomorphism. 

Second, we may assume that $M$ is finitely generated. 
It is easy to see that the functors $\Gamma$ and $\Sh$ commute with filtered colimits. In a locally noetherian Grothendieck category, a filtered colimit of injective objects remains injective. Therefore, by \cite[Proposition~A.4]{gs18incmon}, the right derived functors of $\Gamma$ also commute with filtered colimits. So, for an arbitrary module $N$, we may write $F_{i, N}$ as the colimit of the maps $\{F_{i, N_{\alpha}}\}_{\alpha}$ where $\{N_\alpha\}_{\alpha}$ is the directed system of finitely generated submodules of $N$. Therefore, if we show $F_{i, M}$ is an isomorphism for finitely generated $M$, then $F_{i, N}$ will also be an isomorphism. So for the rest of the proof, we assume that $M$ is a finitely generated torsion-free $R$-module.

We first show that $F_{i,M}$ is injective. Consider a short exact sequence
\begin{displaymath}
0 \to M \to I \to N \to 0 
\end{displaymath}
where $I$ is an injective $R$-module. We obtain a commutative diagram
\[
\begin{tikzcd}
  \Sh(\rR^{i-1}\Gamma(I)) \arrow[r]\arrow[d] & \Sh(\rR^{i-1}\Gamma(N)) \arrow[d] \arrow[r] & \Sh(\rR^i \Gamma(M)) \arrow[d] \arrow[r] & \Sh(\rR^i\Gamma(I)) = 0 \arrow[d] \\
  \rR^{i-1}\Gamma(\Sh(I)) \arrow[r] & \rR^{i-1}\Gamma(\Sh(N)) \arrow[r] & \rR^{i}\Gamma(\Sh(M)) \arrow[r] & \rR^i\Gamma(\Sh(I)) 
\end{tikzcd}
\]
where the rows are exact. The first two vertical arrows are isomorphisms (by induction), and the fourth vertical map is injective, so by the four lemma, the third vertical map, which is $F_{i, M}$ is also injective. This implies that the canonical map $\Sh^l\rR^i\Gamma(M) \to \rR^i\Gamma(\Sh^l(M))$ is also injective for all $l \in \bN$.

Let $\kappa_l(M) = \ker[\rR^i \Gamma(M) \to \Sh^l(\rR^i \Gamma(M))]$. We will now show that the modules $\Sh\kappa_l(M)$ and $\kappa_l(\Sh(M))$ are isomorphic for $l \gg 0$. For all $l$, we have a short exact sequence
\begin{displaymath}
 0 \to M \to \Sh^l(M) \to \De_l(M) \to 0.
\end{displaymath}
Therefore, we get the exact sequence
\begin{displaymath}
\rR^{i-1}\Gamma(\Sh^l(M)) \to \rR^{i-1}\Gamma(\De_l(M)) \to \rR^i \Gamma(M) \to \rR^i \Gamma(\Sh^l(M)).
    \end{displaymath}
{Let $I^{\bullet}$ and $J^{\bullet}$ be injective resolutions of $M$ and $\Sh^l(M)$ respectively}. The induced map $I^{\bullet} \to J^{\bullet}$ is easily seen to factor via $I^{\bullet} \to \Sh(I^{\bullet}) \to J^{\bullet}$. Therefore, the last map in the above exact sequence factors 
\begin{displaymath}
\rR^i \Gamma(M) \to \Sh^l \rR^i \Gamma(M) \to \rR^i \Gamma(\Sh^l(M)). 
\end{displaymath}
We have already shown that $\Sh^l \rR^i \Gamma(M) \to \rR^i\Gamma(\Sh^l(M))$ is injective, so we get that $\kappa_l(M)$ can also be written as $\ker[\rR^i \Gamma(M) \to \rR^i\Gamma (\Sh^l(M))]$ and from the exact sequence above, we also get 
\begin{displaymath}
\kappa_l(M) = \ker[\rR^i \Gamma(M) \to \rR^i\Gamma(\Sh^l(M))] \cong \coker[\rR^{i-1} \Gamma (\Sh^l(M)) \to \rR^{i-1} \Gamma(\De_l (M))].
\end{displaymath}
Therefore, we have isomorphisms
\begin{align*}
\Sh \kappa_l(M) & = \Sh \ker[\rR^i\Gamma(M) \to \rR^i\Gamma(\Sh^l(M))] \\
 &\cong \Sh \coker[\rR^{i-1} \Gamma (\Sh^l(M)) \to \rR^{i-1} \Gamma(\De_l (M))]\\
 &\cong \coker[\Sh \rR^{i-1} \Gamma (\Sh^l(M)) \to \Sh\rR^{i-1}\Gamma(\De_l(M))] &\mbox{($\Sh$ commutes with taking cokernels)}\\
 &\cong \coker[\rR^{i-1} \Gamma (\Sh(\Sh^l(M))) \to \rR^{i-1}\Gamma(\Sh(\De_l(M))) ]&\mbox{(induction hypothesis)}\\
 &\cong \coker[\rR^{i-1} \Gamma (\Sh^l(\Sh(M))) \to \rR^{i-1}\Gamma(\De_l(\Sh(M))) ]&\mbox{($\Sh$ commutes with $\Sh^l$ and $\De_l$)}\\
 &\cong \ker[\rR^i\Gamma(\Sh(M)) \to \rR^i\Gamma(\Sh^l(\Sh(M)))] \\
 &= \kappa_l(\Sh(M))
\end{align*}
proving our claim that $\Sh\kappa_l(M) \cong \kappa_l(\Sh(M))$ for $l \gg 0$. 

Now, since $\rR^i \Gamma(M)$ is a torsion module, we have $\rR^i\Gamma(M) = \colim_{\substack{l \in \bN}} \kappa_l(M)$, using which we get
\begin{displaymath}
\Sh \rR^i \Gamma(M) = \Sh (\colim \kappa_l(M)) \cong \colim \Sh(\kappa_l(M)) \cong \colim \kappa_l(\Sh(M)) \cong \rR^i \Gamma(\Sh(M)).
\end{displaymath}
For the second isomorphism, we have used the fact that $\Sh$ commutes with colimits and for the third isomorphism, that the functor $\Sh$ commutes with $\kappa_l$ for sufficiently large $l$, which we proved in the previous paragraph. Therefore, the natural map $F_{i, M}$ is an isomorphism, as required.
\end{proof}

\begin{remark}
To prove Proposition~\ref{prop:shcommutesderivedgamma}, we have only used the fact that $\Mod_R$ is locally noetherian, that the Serre subcategory $\Mod_R^{\tors}$ satisfies Property (Inj), that a torsion $R$-module $M$ can be realized as the colimit of the kernel of the maps $M \to \Sh^l(M)$, the functors $\Gamma$ and $\Sh$ commute with taking filtered colimits, and that the functor $\Sh$ commutes with $\Gamma$ and $\De_l$.
\end{remark}

\begin{corollary}\label{cor:splitshisdersat}
    Let $M$ be a torsion-free $R$-module. If the natural map $M \to \Sh^l(M)$ splits for all $l$, then $\rR^i\Gamma(M) = 0$ for all $i \ge 0$.
\end{corollary}
\begin{proof}
The proof of \cite[Proposition~4.21]{nag19vi} applies.
\end{proof}

\begin{proposition}\label{prop:semiindiffdersat}
    A finitely generated $R$-module $M$ is semi-induced if and only if $\rR\Gamma(M) = 0$. 
\end{proposition}
\begin{proof}
Assume that $M$ is an induced module. Then, the natural map $M \to \Sh^l(M)$ splits (Lemma~\ref{lem:shiftofinduced}). Therefore, by Corollary~\ref{cor:splitshisdersat}, all the local cohomology of $M$ vanishes. By d\'evissage, the functor $\rR\Gamma$ vanishes on semi-induced modules as well.

Now, assume $\rR\Gamma(M) = 0$. In particular, the $R$-module $M$ is torsion-free as $\Gamma(M)=0$. We will show by induction on $t_0(M)$ that $M$ is semi-induced. When $t_0(M) = 0$, we have $\De(M) = 0$ and therefore $M$ is semi-induced by Proposition~\ref{prop:dmzeroisflat}. Now, assume $t_0(M) > 0$. By the shift theorem, we have a short exact sequence
\begin{displaymath}
0 \to M \to \Sh^l(M) \to \De_l(M) \to 0
\end{displaymath}
with $\Sh^l(M)$ being semi-induced. Therefore $\rR\Gamma(\Sh^l(M))=0$ by the previous paragraph. So from the long exact sequence of local cohomology, we also have $\rR\Gamma(\De_l(M)) = 0$. By Proposition~\ref{prop:basicshift}{(e)}, we see that $t_0(\De_l(M)) < t_0(M)$ and so {$\De_l(N)$} is semi-induced by the induction hypothesis. 
Now, by Corollary~\ref{cor:semises}, we have that $M$ is semi-induced, as required.
\end{proof}

\begin{defn}
We let $\bS \colon \Mod_R \to \Mod_R$ be the \textit{saturation functor} (where $\bS = S \circ T$). The saturation functor $\bS$ is a left exact functor. We let $\eta \colon \id \to \bS$ denote the unit of adjunction. An $R$-module $M$ is \textit{saturated} if the natural map $M \to \bS(M)$ is an isomorphism. An $R$-module $M$ is \textit{derived saturated} if it is saturated and $\rR^i\bS(M) = 0$ for all $i$. 
\end{defn}
The next result holds whenever the Serre subcategory satisfies Property (Inj); see \cite[Proposition~4.6]{ss19gl2}) for details.
\begin{lemma}\label{lem:dersatifflcvan}
An $R$-module $M$ is derived saturated if and only if $\rR\Gamma(M)=0$. 
\end{lemma}
\begin{corollary}
An $R$-module is derived saturated if and only if it is semi-induced.
\end{corollary}
\begin{proof}
This follows by combining Lemma~\ref{lem:dersatifflcvan} and Proposition~\ref{prop:semiindiffdersat}.
\end{proof}

In characteristic zero, Sam--Snowden proved that $\Mod_R^{\tors}$ is equivalent to $\Mod_R^{\gen}$. Recently, Snowden \cite{sno21stable} extended this result to \textit{integral} $\GL$-algebras in characteristic zero, i.e., for a $\GL$-algebra $A$ with $|A|$ an integral domain, Snowden proved that the generic category $\Mod_A^{\gen} = \Mod_A/\Mod_A^{\tors}$ is equivalent to $\Mod_A^{\lf}$, the category of locally finite length $A$-modules. We show that this result fails for $R$ in positive characteristic. 
\begin{proposition}\label{prop:gentorinequiv}
    The categories $\Mod_R^{\gen}$ and $\Mod_R$ are not equivalent abelian categories in positive characteristic.
\end{proposition}
\begin{proof}
In $\Mod_R^{\tors}$, it is easy to see that every finitely generated object has finite injective dimension. We will show that $T(R)$ has infinite injective dimension in $\Mod_R^{\gen}$. Firstly, the injective dimension of $\lw^i(\bV)$ in $\Rep^\pol(\GL)$ is unbounded as $i$ varies (see \cite[Page~6]{tot97gl}). So $R$ has infinite injective dimension in $\Mod_R$ as injective objects in $\Mod_R$ are injective in $\Rep^\pol(\GL)$. Now, since $R$ is {derived saturated}, we may take an injective resolution of $T(R)$, and apply $S$ to obtain an injective resolution of $R$. Therefore the injective dimension of $T(R)$ is at least the injective dimension of $R$, and so infinite as well. This implies that the two categories are not equivalent.
\end{proof}

\begin{theorem}\label{thm:finiteext}
Let $M, N$ be two finitely generated $R$-modules. The module $\ext^i_R(M, N)$ is finite length for all $i$, and vanishes for $i \gg 0$.
\end{theorem}
\begin{proof}
By d\'evissage, it suffices to show that the conclusions of the theorem hold when $M$ and $N$ are the generators of the derived category given in Theorem~\ref{thm:gendbmodr}. First, assume $N$ is the simple $R$-module $L_{\lambda}$. The $R$-module $L_{\lambda}$ has a finite injective resolution by finite length injectives. So $\ext^i_R(M, L_{\lambda})$ vanishes for sufficiently large $i$ and is finite length for all $i$ (since $M$ is finitely generated). 

Next, assume $N$ is the induced $R$-module $R \otimes L_{\lambda}$ and $M$ is the simple $R$-module $L_{\mu}$. The module $N$ is derived saturated (Proposition~\ref{prop:semiindiffdersat}), so $\ext^i_R(L_{\mu}, R \otimes L_{\lambda})$ vanishes for all $i$ (\cite[Proposition~4.7]{ss19gl2}). Finally, assume $M = R \otimes L_{\mu}$ and and $N = R\otimes L_{\lambda}$ are both induced modules. The $\GL$-representation $L_{\mu}$ has a finite projective resolution $P_{\bullet}$ in $\Rep^{\pol}(\GL)$, see e.g., \cite[Theorem~1]{tot97gl}. The complex $R \otimes P_{\bullet}$ is a finite projective resolution of $R \otimes L_{\lambda}$. Therefore, the conclusion of the theorem holds in this case as well.
 \end{proof}

\subsection{Semi-orthogonal decomposition} \label{ss:semiorth}
Using Property (Inj), we obtain a semi-orthogonal decomposition
\begin{displaymath}
\rD^{+}(\Mod_R) = \langle \rD^{+}(\Mod_R^{\tors}), \rD^{+}(\Mod_R^{\gen}) \rangle
\end{displaymath}
of the derived category $\rD^{+}(\Mod_R)$ (see below for definitions and \cite[Proposition~4.9]{ss19gl2} for details). In this section, we prove the analogous result for $\rD^b_{\fgen}(\Mod_R)$ (Theorem~\ref{thm:semiorth}). We emphasize that for $\rD^b_{\fgen}(\Mod_R)$, the semi-orthogonal decomposition is not a formal consequence of Property (Inj). It will follow from our next theorem on the finiteness of local cohomology, which we prove using the results of the previous two sections.
\begin{theorem}[Finiteness of local cohomology]\label{thm:finitelc}
For a finitely generated $R$-module $M$, the $R$-module $\rR^i\Gamma(M)$ is finite length for all $i$, and is zero for sufficiently large $i$.
\end{theorem}
\begin{proof}
For torsion modules and semi-induced modules, the conclusions of the theorem hold by Corollary~\ref{cor:lcvanishingtorsion} and Proposition~\ref{prop:semiindiffdersat} respectively. Since the torsion and semi-induced modules generate $\rD^b_{\fgen}(\Mod_R)$, the result holds for all finitely generated $R$-modules.
\end{proof} 
We now recall the definition of a semi-orthogonal decomposition of a triangulated category. We refer the reader to \cite[Section~4]{ss19gl2} for a succinct reference.
\begin{defn}
A \textit{semi-orthogonal decomposition} of a a triangulated category $\cT$ is a sequence $\langle \cA_1, \cA_2, \ldots, \cA_n \rangle$ of full triangulated subcategories such that 
\begin{enumerate}
\item $\Hom_{\cT}(M_i, M_j) = 0$ for objects $M_i \in \cA_i, M_j \in \cA_j$ with $i < j$, and
\item the smallest triangulated subcategory containing $\cA_1, \cA_2, \ldots, \cA_n$ is $\cT$.
\end{enumerate}
\end{defn}
\begin{theorem}\label{thm:semiorth} We have a semi-orthogonal decomposition
\begin{displaymath}
\rD^b_{\fgen}(\Mod_R) = \langle \rD^b_{\fgen}(\Mod_{R}^{\tors}),\rD^b_{\fgen}(\Mod_{R}^{\gen}) \rangle.
\end{displaymath}
where $\rD^b_{\fgen}(\Mod_{R}^{\gen})$ is identified as a subcategory of $\rD^b_{\fgen}(\Mod_R)$ using the functor $\rR S$. 
\end{theorem}
\begin{proof}
This is a consequence Theorem~\ref{thm:finitelc} (see \cite[Proposition~4.15]{ss19gl2} for more details).
\end{proof}

We can now prove a generalization of Proposition~\ref{prop:triangle}.
\begin{proposition}\label{prop:triangledb}
    Given an object $M$ in $\rD^b_{\fgen}(\Mod_R)$, there exists a distinguished triangle $\rR\Gamma(M) \to M \to \rR\bS(M) \to$ in $\rD^b_{\fgen}(\Mod_R)$.
\end{proposition}
\begin{proof}
This follows by combining \cite[Proposition~4.6]{ss19gl2} with Theorem~\ref{thm:finitelc}.
\end{proof}

\bibliographystyle{plain}
\bibliography{ref}

\end{document}